\documentclass[a4paper,11pt]{amsart}

\usepackage[english]{babel}

\usepackage[a4paper,left=30mm,right=30mm,top=35mm,bottom=35mm,marginpar=25mm]{geometry}

\usepackage{graphics} 

\usepackage{epsfig}
\usepackage{color}

\usepackage{amsmath}    
\usepackage{amsfonts,amssymb}
\usepackage{amsthm}
\usepackage{graphicx}
\usepackage{esint}
\usepackage{stackrel}
\usepackage{enumerate}

\usepackage[colorlinks=true, citecolor=blue]{hyperref} 

\numberwithin{equation}{section}
\newtheorem{theorem}{Theorem}[section]
\newtheorem{lemma}[theorem]{Lemma}
\newtheorem{proposition}[theorem]{Proposition}
\newtheorem{corollary}[theorem]{Corollary}

\theoremstyle{definition}
\newtheorem{definition}[theorem]{Definition}
\newtheorem{remark}[theorem]{Remark}

\newtheorem{Convention}[theorem]{Convention}

\newcommand{\R}{\mathbb{R}}

\newcommand{\A}{\mathcal{A}}
\newcommand{\G}{\mathcal{G}}
\newcommand{\F}{\mathcal{F}}

\newcommand{\1}{\mathbf{1}}

\newcommand{\eps}{\varepsilon}

\newcommand{\Hf}{\mathcal{H}}
\newcommand{\e}{\mathbf{e}}
\newcommand{\f}{\mathbf{f}}
\newcommand{\p}{\mathbf{p}}
\newcommand{\q}{\mathbf{q}}
\newcommand{\C}{\mathbf{C}}
\newcommand{\D}{\mathbf{D}}

\DeclareMathOperator{\Id}{Id}

\DeclareMathOperator{\spt}{spt}
\DeclareMathOperator{\Div}{div}
\DeclareMathOperator{\Sym}{Sym}
\DeclareMathOperator{\Lip}{Lip}

\DeclareMathOperator{\dist}{dist}

\title[Regularity of minimizers for a model of charged droplets]{Regularity of minimizers for a model of charged droplets}

\author[G. De Philippis]{Guido De Philippis}
\address{G.D.P.: SISSA, Via Bonomea 265, 34136 Trieste, Italy}
\email{guido.dephilippis@sissa.it}

\author[J.~Hirsch]{Jonas Hirsch}
\address{J.H.:  Mathematisches Institut, Universit\"at Leipzig, Augustus Platz 10, D04109 Leipzig, Germany}
\email{hirsch@math.uni-leipzig.de}

\author[G. Vescovo]{Giulia Vescovo}
\address{G.V.: SISSA, Via Bonomea 265, 34136 Trieste, Italy}
\email{gvescovo@sissa.it, giulia.vescovo88@gmail.com}

\subjclass[2010]{}

\begin{document}
\begin{abstract} We investigate properties of minimizers of  a variational model describing the shape of charged liquid droplets.
The model, proposed by Muratov and Novaga, takes into account the regularizing effect due to the screening of free counterionions in the droplet. In particular  we 
 prove partial regularity of  minimizers, a first step toward the understanding of further  properties of minimizers. 
\end{abstract}

\maketitle

\section{Introduction}

\subsection{Background and description of the model}
In this paper we investigate the regularity of minimizers for  a variational model describing the shape of charged liquid droplets. Roughly speaking, 
the shape of a charged  liquid droplet is determined  by the  competition between an   ``aggregating''  term, due to  surface tension forces, and  to a ``disaggregating''  term due to  the repulsion effect between  charged particles.

Several  models  proposed in  literature are based on this principle. Among them, one of the simplest and most used assumes that charged  droplets are  stationary points for the  following free energy: 
\begin{equation}\label{energy}
P(E )+\frac{Q^{2}}{\mathcal C(E)}.
\end{equation}
Here, \(E\subset \R^{3}\) corresponds to the volume occupied by the droplet,  \(P(E)\) is its perimeter, \(Q\) is the total charge and
\begin{equation}\label{e:cap}
\frac{1}{\mathcal C(E)}:=\inf\Bigg\{\frac{1}{4\pi} \iint \frac{d\mu(x)d\mu(y)}{|x-y|}: \spt \mu\subset E, \mu(E)=1\Bigg\},
\end{equation} 
 takes into account the repulsive forces between  charged  particles. Note that  \(\mu\) can be though as a (normalized) density of charges and that \(\mathcal C (E)\) is the classical Newtonian capacity of the set \(E\). In particular one assumes that the optimal shapes are given by the following variational problem:
\begin{equation}\label{e:ray}
\min_{|E|=V} P(E )+\frac{Q^{2}}{\mathcal C(E)}.
\end{equation}
Heuristically,  one expects the perimeter term to dominate for small value of the charge \(Q\) thus forcing the droplet to have a spherical or almost spherical shape, while the repulsion  term should  become  dominant  for large values of \(Q\), thus leading  to the formation of singularities and/or to the ill-posedness of~\eqref{e:ray}. This heuristics is confirmed by the perturbative analysis of~\eqref{energy} around a spherical shape. This computation, performed for the first time by Lord Rayleigh in 1882, ~\cite{Rayleigh82}, shows that the spherical droplet  is linearly stable only for \(Q\) smaller than a critical threshold. This is  known as the Rayleigh criterion. 

The transition from a stable to an unstable behavior of spherical droplets has also been verified experimentally, starting from the work of  Zeleny at the beginning of  1900~\cite{Zeleny17} (in a slightly different context). More precisely,  it has been observed that a spherical droplet  exposed to an electric field, remains stable until the total charge is below  a critical value $Q_{c}>0$, while,  as soon as $Q$ exceeds $Q_{c}$ the droplet changes its appearance and the surface start to develop singularities, the so called Taylor's cones,~\cite{Taylor64}. Whenever $Q\ge Q_{c}$ a very thin steady jet composed by small but highly charged little balls is formed,~\cite{WilsonTaylor25,DoyleMoffettVonnegut64,RichardsonPiggHightower89,DuftAchtzehnMullerHuberLeisner03}.

In spite of the interest of~\eqref{e:ray} in applications, a rigorous mathematical study of this model  has been only  performed in the last years, mostly thanks to the work of Goldman, Muratov, Novaga and Ruffini, see~\cite{GoldmanNovagaRuffini15,MuratovNovaga16,GoldmanRuffini17,MuratovNovagaRuffini18,GoldmanNovagaRuffini18} and references therein. 

The starting point of their analysis is the following remarkable and somehow disappointing observation: \emph{Problem~\eqref{e:ray} is always ill-posed}. More precisely, in \cite{GoldmanNovagaRuffini15}, it is shown   that 
\[
\inf_{|E|=V}P(E)+\frac{Q^{2}}{\mathcal C(E)}=P(B^{V}),
\]
where \(B^{V}\) is the ball of volume \(V\). Since \(B^{V}\) is a competitor  for the variational problem , this clearly implies that there are no minimizers of~\eqref{e:ray}. 

The above equality  is obtained by constructing a minimizing sequence consisting of a ball of roughly volume \(V\) together with several balls with vanishing  perimeter and volumes  and very high charge escaping at infinity. Hence, on the mathematical side, the phenomena observed by Zeleny appears for  \emph{every  value of the charge}. Let us also remark that ill-posedness of~\eqref{e:ray}  is shown also if one assumes that all the set involved in the minimization problem are a-priori bounded,  \cite[Theorem 1.3]{GoldmanNovagaRuffini15}.

It then becomes natural to investigate the local minimality of the ball, at least for ``small'' perturbations and small values of \(Q\). In \cite[Theorems 1.4 \& 1.7]{GoldmanNovagaRuffini15}  the linear stability of the ball in the small charge regime, is  upgraded to  local minimality in a sufficiently strong topology.  On the other hand  Muratov and Novaga showed  that the ball is \emph{never} a local minimizer of~\eqref{e:ray} under (smooth) perturbation which are small in \(L^{\infty}\), \cite[Theorem 2]{MuratovNovaga16}. 
We also refer the reader to \cite{GoldmanRuffini17} where well-posedness is recovered under suitable geometric restrictions and to~\cite{MuratovNovagaRuffini18} for the case of ``flat'' droplets.

The main phenomena driving to the ill-posedness of~\eqref{e:ray} is the possibility of concentrating a high  charge on small volumes.  In order to avoid this  situation, in \cite{MuratovNovaga16},  Muratov and Novaga proposed  as a possible regularization mechanism the finite screening length in the conducting liquid , by introducing the entropic effects associated with the presence of free ions in the liquid, see also~\cite{Deserno01,ThaokarDeshmukh10} for a related model. They  suggested to  consider the following  \textsl{Debye-H\"uckel-type free energy}  (in every dimension)
\begin{equation}\label{e:DH}
\mathcal{F}(E,u,\rho):=P(E)+Q^{2}\Bigg\{\int_{\mathbb{R}^{n}}a_{E}|\nabla{u}|^{2}\,dx+K\int_{E}\rho^{2}\,dx\Bigg\}.
\end{equation}
Here 
\[
a_{E}(x):=\1_{E^{c}}+\beta\1_{E},
\]
 where \(\1_{F}\) is the characteristic function of a set  \(F\) and $\beta>1$ is the permittivity of the droplet. The (normalized) density of  charge \(\rho\in L^{2}(\mathbb{R}^{n})$ satisfies 
\begin{equation}\label{e:rho}
\rho \1_{E^c}=0 \qquad\text{and}\qquad \int \rho=1,
\end{equation}
and the  electrostatic potential  \(u\) is such that  $\nabla u\in L^{2}(\mathbb{R}^{n})$ and 
\begin{equation} \label{vincolo1}
-\Div \big(a_{E}\,\nabla{u}\big)=\rho\qquad\text{in}\;\mathcal{D}'(\mathbb{R}^{n}).
\end{equation}
 $K>0$  is a physical constant related to the model.\footnote{Actually in \cite{MuratovNovaga16},  the energy~\eqref{e:DH} is written as 
 \[
 \sigma P(E)+Q^{2}\Bigg\{\frac{\beta_{0}}{2}\int_{\mathbb{R}^{n}}a_{E}|\nabla{u}|^{2}\,dx+K\int_{E}\rho^{2}\,dx\Bigg\}
 \]
 for suitable parameters  \(\sigma\) and \(\beta_{0}\) and the relation~\eqref{vincolo1} is replaced by \(-\beta_{0}\Div \big(a_{E}\,\nabla{u}\big)=\rho\). However it is easy to see that the parameters \(\sigma\) and \(\beta_{0}\) can be absorbed in \(Q\) and \(K\), see also the discussion below.}

The variational model proposed in \cite{MuratovNovaga16}, where one assumes  \emph{a-priori}  that all the sets are contained in a fixed (large)  ball \(B_{R}\), is the following 
\begin{equation}\label{e:variational}
\min \big\{\F(E,u,\rho):|E|=V, E\subset B_{R}, (u,\rho)\in \A(E)\big\},
\end{equation}
where we have set
\begin{equation}\label{e:admissible}
\mathcal{A}(E):=\big\{\text{\((u,\rho) \in D^{1}(\R^{n})\times  L^2(\R^n)\): \(u\) and \(\rho\) satisfy~\eqref{vincolo1} and~\eqref{e:rho}}\big\},
\end{equation}
and 
\[
D^{1}(\R^{n})=\overline{ C_{c}^{\infty}(\R^{n})}^{\mathring{W}^{1,2}(\R^{n})}\qquad  \|\varphi\|_{\mathring{W}^{1,2}(\R^{n})}=\|\nabla \varphi\|_{L^{2}(\R^{n})}.
\]
Note that the class of admissible couples \(\mathcal A(E)\) is non-empty only if \(n\ge 3\), for this reason \emph{this assumption will be in force throughout the whole paper}, see also Remark \ref{rmk:dimension}.

Thanks to the a-priori boundedness assumption $E\subset B_{R}$, existence of a minimizer in the class of sets of finite perimeter  can  be easily shown, see \cite[Theorem 3]{MuratovNovaga16}. 

Note that the presence of the \(L^{2}\) norm of \(\rho\) in the energy is exactly what prevents the concentration of charges. Indeed, if one assumes that \(\beta=1\) so that~\eqref{vincolo1} reduces to 
\[
-\Delta u=\rho,
\]
then the minimization problem~\eqref{e:variational} can be written, in dimension \(n=3\)   as 
\[
\min_{|E|=V, E \subset B_{R}}  P(E)+Q^{2}\min \Biggl\{\frac{1}{4\pi }\iint \frac{\rho(x)\rho(y) dxdy}{|x-y|}+K\int \rho^{2}  \text{ s.t. \(\rho \1_{E^c}=0, \int \rho=1\)}\Biggr\},
\]
which should be compared with~\eqref{e:cap} and~\eqref{e:ray}. In view of this we also note that, on the mathematical ground,  the variational problem~\eqref{e:variational} can also be considered as an ``interpolation'' between the classical Otha-Kawasaki problem, and the free-interfaces problems arising in optimal design studied for instance in~\cite{AmbrosioButtazzo93,Lin93,De-PhilippisFigalli15a, FuscoJulin15}.

\subsection{Main results}  
Once existence of a minimizers of~\eqref{e:variational}  is obtained it is natural to investigate  their  qualitative and quantitative properties, also to understand to which extent the predictions of model agree with the observed phenomenology. In particular the following questions arise, compare with \cite{MuratovNovaga16}:
\begin{enumerate}[-]
\item Is every minimizers smooth, at least outside a small singular set? 
\item Which is the structure of  (possible) singularities  of minimizers? Do they agree with Taylor's cones\footnote{Note that this is possible only if \(\beta\) is large compared to \(1\), see the discussion at the end of this introduction and Remark \ref{rmk:cordes}}?
\item Is it possible to show existence/non-existence  of minimizers removing the a-priori confinement assumption?
\item Can one show that for small value of the charges minimizers of~\eqref{e:variational} are balls in agreement with  experimental  observations?
\end{enumerate}

 In this paper we address the question of regularity of minimizers. Our main result is the following partial regularity theorem:

\begin{theorem}\label{thm:main}
Let \(n\ge 3\) and \(B>0\). Then there exists \(\eta=\eta(n,B)>0\) with the following property:  if  \(E\) is  a minimizer of~\eqref{e:variational} with \(\beta \le B\) then there exists a closed set \(\Sigma_{E}\subset \partial E\) such that \(\Hf^{n-1-\eta}(\Sigma_{E})=0\) and \(\partial E\setminus \Sigma_{E}\) is a \(C^{1,\vartheta}\) manifold for all \(\vartheta\in (0,1/2)\).
\end{theorem}

As it is customary in Geometric Measure Theory, the proof Theorem~\ref{thm:main} is based on an \(\varepsilon \)-regularity result which is interesting  on  its own. 
In order to keep track of the various dependence on the parameters let us first fix some notations, which will be useful also in the sequel. For  \(E\subset \R^{n}\) we define
\begin{equation}\label{e:G}
\G_{\beta,K}(E):=\inf_{(u,\rho)\in\mathcal A(E)}\left\{\int_{\mathbb{R}^{n}}a_{E}|\nabla{u}|^{2}+K\int_{E}\rho^{2}\right\},
\end{equation}
where the set of admissible pairs  \(\A(E)\) is defined in~\eqref{e:admissible} (if the dependence on the parameter is not relevant we will simply write \(\G\)). Since 
 \[
 (u,\rho) \in \mathcal{A}(E)\quad \Longrightarrow \quad \Bigl(\lambda^{2-n} u\Bigl(\frac{\cdot}{\lambda}\Bigr), \lambda^{-n} \rho\Bigl(\frac{\cdot}{\lambda}\Bigr)\Bigr) \in \mathcal{A}(\lambda E),
 \]
 one has 
\begin{equation*}\label{e:Gscaling}
 \G_{\beta,\lambda ^{2} K}(\lambda E)=\lambda^{2-n}\G_{\beta, K}( E).
\end{equation*}
 Setting
\begin{equation*}
\mathcal F_{\beta,K,Q}(E):= P(E)+Q^{2}\G_{\beta,K}(E),
\end{equation*}
one gets
\[
\mathcal F_{\beta, K,Q}(E)= \lambda^{1-n} \F_{\beta, K\lambda^{2},Q\lambda^{\frac{2n-3}{2}}}(\lambda E).
\]
In particular, by replacing \(K\) and \(Q\) with \(K(\omega_{n}/V)^{\frac{2}{n}}\) and \(Q(\omega_{n}/V)^{1-\frac{3}{2n}}\) we can assume that  \(V=|B_{1}|=:\omega_{n}\).  Namely, for \(R\ge 1\) we will consider the following problem
\begin{equation} \label{e:problem} \tag{$\mathcal{P}_{\beta,K,Q,R}$}
 \min \big\{\F_{\beta, K,Q}(E):|E|=|B_{1}|,\  E\subset B_{R}\big\}.
\end{equation}
 Furthermore, given a set of finite perimeter \(E\) we  define the  \emph{spherical excess} at a point \(x\) and at scale \(r>0\)  as 
\[
\e_E(x,r):=\inf_{\nu \in \mathbb S^{n-1}} \frac{1}{r^{n-1}} \int _{\partial^{*}E\cap B_{r}(x)}\frac{|\nu_{E}(y)-\nu|^{2}}{2}d \Hf^{n-1}(y),
\]
where, \(\partial^{*}E\) is the reduced boundary of \(E\), $\nu_{E}$ is the measure-theoretic unit normal to $\partial E$, see \cite{Maggi12}, and  \(B_{r}(x)\) is the ball of center \(x\) and radius \(r\). We also define  the   \emph{normalized Dirichlet energy} as
\[
D_E(x, r):=\frac{1}{r^{n-1}} \int_{B_{r}(x)} |\nabla u_{E}|^{2},
\]
where \(u_{E}\) is the minimizer in~\eqref{e:G}, whose existence and uniqueness can be easily proved, see Proposition~\ref{firstreg} below.  With  these conventions, the \(\varepsilon\)-regularity results can be stated as follows, see also Theorem \ref{thm:mainmeglio} below for a slightly more precise statement,
\begin{theorem} \label{thm:maineps} 
Given \(n\ge 3\), \(A>0\) and \(\vartheta\in (0,1/2)\), there exits \(\varepsilon_{\textnormal{reg}}=\varepsilon_{\textnormal{reg}}(n, A, \vartheta)>0\)  such that if \(E\) is minimizer of\(~\eqref{e:problem}\) with \(Q+\beta+K+\frac{1}{K}\le A\), \(x\in \partial E\)  and 
\[
r+\e_E(x, r)+D_E(x, r)\le \varepsilon_{\textnormal{reg}},
\] 
then \(E\cap B(x,r/2)\)  coincides with  the  epi-graph of a \(C^{1,\vartheta}\) function. In particular  \(\partial E\cap B(x,r/2)\) is a \(C^{1,\vartheta}\) \((n-1)\)-dimensional manifold.
\end{theorem}
Let  us conclude this section with some remarks: 

First beside its intrinsic interest, combining   Theorem~\ref{thm:maineps} with the analysis of the linearized energy around a ball one can show show that the balls uniquely minimize~\eqref{e:problem} for small value of \(Q\). This will be addressed in a forthcoming paper.

Second we  note   that  the dimension of the singular set in Theorem~\ref{thm:main} depends only on the gap between the two permittivity constants  and not on the other parameters appearing in the model. On the other hand the ``regularity scale'' in Theorem~\ref{thm:maineps} depends on all the parameters involved. A similar fact has been observed in the context of free interfaces models in~\cite{De-PhilippisFigalli15a, FuscoJulin15}.

Finally, it seems reasonable to expect  that  \(C^{1,\vartheta}\) regularity  of \(\partial E\) can be upgraded to \(C^{\infty}\) smoothness by some bootstrap argument. We leave this interesting question open.

\subsection{Strategy of the proof and  structure of the paper} Though the energy we are considering has a certain similarity with those studied in optimal design problems, the fact that the minimization problem in definition~\eqref{e:G} is performed only among admissible pairs \((u,\rho)\in \A(E)\) makes very difficult to make  local perturbations. In particular, problem~\eqref{e:problem} has (a priori) no local scaling invariance. For this reason in Section~\ref{s:G} we study carefully the energy \(\G(E)\) and its minimizers \((u_{E}, \rho_{E})\). Moreover we establish boundedness of \(u_{E}\) and \(\rho_{E}\) . 

In order to study the regularity of minimizers one needs to perform local variations and hope that these gives localized (or almost localized) changes of the energy.  This is not completely evident due to the presence of a volume constraint and  to the non-local character of \(\G\). As, it is well known, the volume constraint  can be relaxed into a ``perturbed'' minimality property of minimizers. In order to have estimates uniform in the structural parameters it will be important to have this ``perturbed'' minimality property uniform in the class of minimizers. In Section~\ref{s:EL} we start studying how the energy varies according to a flow of diffeomorphism, which will be important in performing small volume adjustments and  we  establish the Euler Lagrange equations for minimizers. In Section~\ref{s:lambda} we prove the perturbed minimality property and we study the behavior of the energy under local perturbations. In Section~\ref{s:comp} we prove  the compactness of the class of minimizers in the \(L^1\) topology, which though not used in the proof of our main results is interesting by its own.

The  next  step consists in establishing local perimeter and volume  estimates for the minimizers of~\eqref{e:problem}. Usually these estimates are easily obtained by combining minimality with local isoperimetric inequalities. Here,  due   to the non-local character of the energy term \(\G(E)\) and the absence of a natural scaling invariance of the problem, more refined arguments are required. In particular we will first show that the energy \(\G\) is monotone by set inclusion. This implies that \(E\) is an outer minimizer for the perimeter and  leads to upper perimeter bounds and lower density estimates for \(E^{c}\). Estimating the density of \(E\) is instead more complicated and requires to perform an inductive argument showing that if \(E\) has small relative measure in a ball \(B_{r}(x)\), then the Dirichlet energy of \(u_E\)  decays enough to preserve this information at smaller scales,  leading to a contradiction. In doing this,  higher integrability of the gradient of minimizers of \(\G\) plays a key role.  Local density estimates are obtained in Section~\ref{s:Dir} together with  the boundedness of \(D_E(x,r)\). This  fact combined with the local density and perimeter estimates allow  somehow to  recover the scaling invariance of the problem.

The main step of the proof of Theorems~\ref{thm:main} and ~\ref{thm:maineps} is the decay of the excess established in Section~\ref{s:decay}. Once the local scaling invariance of the problem is recovered, the  proof Theorem~\ref{thm:maineps}  follows the classical De Giorgi's idea of harmonic approximation. Namely we will show that in the regime of small excess and small normalized Dirichlet energy, \(\partial E\) can be well approximated by the graph of a function with ``small'' laplacian. This leads to the decay of the excess which,  thanks to the higher integrability of \(\nabla u_{E}\), in turn also implies the decay of the  normalized  Dirichlet energy and eventually  allows to conclude the proof. 

In Section~\ref{s:thmmain} we prove Theorems~\ref{thm:main} and~\ref{thm:maineps}. Theorem~\ref{thm:maineps}  will be an immediate consequence of Theorem~\ref{excessimprove} (see also  Theorem~\ref{thm:mainmeglio} for a more quantitative version). Theorem~\ref{thm:main} is proved by following the strategy of \cite{FuscoJulin15} where one combines the the \(\eps\)-regularity result with the higher integrability of the \(\nabla u_E\) and the classical regularity theory for minimal surfaces.

Let us remark that most of the above described difficulties arises only in the case when \(\beta\) is relatively large compared to \(1\). Indeed in the regime \(\beta-1\ll1\),  Cordes estimates, see~\cite{Cordes56}, imply that \(\nabla u_{E}\) belongs to \(L^{p}\) with \(p\) large. In this case H\"older inequality immediately gives that the energy term \(\G\) is lower order with respect to the perimeter at small scales.  \(E\) will then  be an \(\omega\)-\textsl{minimizer} of the perimeter and the regularity theory follows for instance from~\cite{Tamanini82},  see Remark \ref{rmk:cordes}. In particular in this case  one obtains full regularity in \(n=3\), thus excluding the formation of Taylor's cone singularities. This phenomena was already  observed in~\cite{StoneListerBrenner98} for a different model of charged droplets. 

 \subsection*{Acknowledgements}
The work of  G.~D.~P. and of G.~V. is supported by the INDAM-grant ``Geometric Variational Problems".

\section{Properties of minimizers of \(\G\)}\label{s:G}

In this section we start establishing some basic properties of minimizers of \(\G\).  We start with the following easy lemma. Here and in the following let \(2^*:=2n/(n-2)\) (recall that we are always working with \(n\ge 3\)).

\begin{lemma}\label{lm:ex}
Let \(n\ge 3\),  \(\beta>1\) and let  \(A:\R^{n}\to \Sym_{n}(\R^n)\) be a symmetric matrix valued function such that  
\[
\Id\le A(x)\le \beta \Id\qquad\text{for all \(x\in \R^n\).}
\]
Then for every \(\rho\in L^{(2^*)'}\) (i.e. the dual of \(L^{2^*}\)) there exists an unique  \(u\in D^1(\R^n)\) such that 
\[
-\Div(A \nabla u)=\rho.
\]
\end{lemma}

\begin{proof}
Recall that for \(u\in D^1(\R^n)\) one has the following Sobolev inequality
\[
\|u\|_{2^*}\le S(n)\|\nabla u\|_{L^2}.
\]
In particular by the assumptions on \(\rho\) and \(A\) the energy
\[
\mathcal E(v)=\frac{1}{2} \int A \nabla v\cdot \nabla v-\int \rho v,
\]
is finite. By Young's inequality \(\mathcal E(v) \) is bounded from below by \(\|\nabla v\|_{L^2}^2-C(n)\|\rho\|_{(2^*)'}^2\). Direct methods of the calculus of variations imply the existence of a unique minimizers which is the desired solution.  Furthermore for the solution we have 
\[ \min_{v \in D^1(\R^n)} \mathcal E(v) =\mathcal E(u) = -\frac12 \int A \nabla u \nabla u = - \frac12  \int \rho u\,.\]
\end{proof}

\begin{remark}\label{rmk:dimension}
In dimension \(n=2\) the above lemma is easily seen to be false, indeed even for a smooth and compactly supported \(\rho\), the solution of 
\[
-\Delta u=\rho,
\]
does not in general satisfy \(\nabla u\in L^2\).
\end{remark}

By the above lemma , if  \(|E|<\infty\) the couple \((u,\rho)\) defined by 
\[
\rho=\frac{\1_{E}}{|E|},\qquad -\Div(a_E\nabla u)=\rho,
\]
is admissible,  \((u,\rho)\in\mathcal{A}(E)\). By testing the equation  by \(u\) and using the Sobolev embedding, we then get 
\[
\int_{\R^n}a_E|\nabla u|^2=\fint_E u\le\Biggl( \fint_E u^{2^*}\Biggr)^{\frac{1}{2^*}} \le \frac{S(n)}{|E|^{\frac{1}{2^*}}}\|\nabla u\|_{2}.
\]
In particular (recall \(\beta>1\))
\begin{equation}
\label{e:gbound}
\G(E)\le \int a_E|\nabla u|^2+K\int \rho^2\le C(n,\beta, K, 1/|E|).
\end{equation}

\begin{proposition} \label{firstreg}
Let $E\subset\mathbb{R}^{n}$ be a set of finite measure. Then there exists an unique pair \((u_E,\rho_E)\in \A(E)\) minimizing \(\G_{\beta,K}(E)\). Moreover
\begin{equation}\label{e:sumconst}
u_E+K\rho_E=\G_{\beta,K}(E)\qquad \textrm{in \(E\)} 
\end{equation}
and 
\begin{equation}\label{e:bound}
0\le u_E\le \G_{\beta,K}(E)\qquad \textrm{and}\qquad  0\le K\rho_E \le \G_{\beta,K}(E)\1_E.
\end{equation}
In particular \(\rho_E\in L^p\) for all \(p\in [1,\infty]\) and 
\begin{equation}\label{e:rhobound}
\|\rho_E\|_{p}\le C(n,\beta, K, 1/|E|).
\end{equation}
\end{proposition}
\begin{proof}
Existence of a  minimizer is an immediate application of the Direct methods in  the Calculus of Variations. Uniqueness follows from  the convexity of the admissible set  \(\A(E)\) and of the strict convexity of the energy
\[
(u,\rho)\mapsto \int a_E|\nabla u|^2+K\int \rho^2.
\]
Let now \(\psi\in C_c^\infty(\R^n)\) be such that 
\begin{equation}\label{e:psi}
\psi\1_{E^c}=0, \qquad\int \psi=0.
\end{equation}
Let \(v\in D^1(\R^n)\) be the solution of 
\begin{equation}\label{e:ev}
-\Div (a_E \nabla v)=\psi.
\end{equation}
If \((u_E,\rho_E)\) is the minimizing pair then \((v_\eps,\rho_\eps)=(u_E+\eps v, \rho_E+\eps \psi)\in \mathcal A(E)\) is admissible. Hence, by taking the derivative with respect to \(\varepsilon\) of its energy we get
\[
0=\int a_E \nabla u_E\nabla v+K\int \rho_E\psi\overset{\eqref{e:ev}}{=}\int(u_E+K\rho_E)\psi.
\]
Since this holds for all \(\psi\) satisfying~\eqref{e:psi} we get that \(u_E+K\rho_E=\textrm{const}\) in \(E\). By multiplying this equation by \(\rho_E\) and integrating over \(E\) we infer that the constant shall be equal to \(\G(E)\), and this proves~\eqref{e:sumconst}. In particular \(u_E\) solves
\begin{equation}\label{e:uu}
-\Div (a_E\nabla u_E)=\frac{\G(E)-u_E}{K}\1_E.
\end{equation}
By testing the above with \( (\G(E)-u_E)_-=-\min\{0, \G(E)-u_E\}\)  we obtain
\[
0=\int_{\{\G(E)<u_E\}}a_{E}\,|\nabla u_E|^{2}\,dx+\int_{\{\G(E)<u\}}\frac{(\G(E)-u_E)^{2}}{K}\,dx,
\]
which implies  the second half of the first inequality in~\eqref{e:bound}. Testing~\eqref{e:uu}  with $u_- = - \min\{0, u \}$ we obtain the first half. The second inequality in~\eqref{e:bound} follows now from the first and~\eqref{e:sumconst}. Inequality \eqref{e:rhobound} follows from \eqref{e:gbound}.
\end{proof}

We establish now the monotonicity of \(\G\) with respect to set inclusion. We start from the following lemma.

\begin{lemma} \label{D_Lemma_1}
Let  $A$, $B:\R^{n}\to \Sym_{n}(\R^n)$ two symmetric matrix valued functions such that $\Id\le A(x)\leq B(x)$ for all \(x\in \R^n\). If \(\rho\in L^{(2^*)'}(\R^n)\)   and $u, v\in D^1(\R^n)$ are the unique solutions of 
\begin{equation} \label{div}
-\Div (A\nabla u)=\rho\qquad\text{and}\qquad -\Div (B\nabla v)=\rho, \qquad\text{in \(\mathcal{D}'(\mathbb{R}^{n})\),}
\end{equation}
then 
\begin{equation}\label{e:resto}
2\int_{\R^{n}}(B-A)\,\nabla v\cdot\nabla v \,dx+\int_{\mathbb{R}^{n}}B\, \nabla{v}\cdot \nabla{v} \,dx\le       \int_{\mathbb{R}^{n}}A\, \nabla{u}\cdot \nabla{u} \,dx.
\end{equation}
In particular
\[
\int_{\mathbb{R}^{n}}B\, \nabla{v}\cdot \nabla{v} \,dx\le  \int_{\mathbb{R}^{n}}A\,\nabla{u}\cdot \nabla{u} \,dx.
\]

\end{lemma}
\begin{proof} Let $\mathcal{E}_{A}$ and $\mathcal{E}_{B}$ be the following functionals defined on $D^1(\R^n)$:
\begin{equation*}
\begin{split}
&\mathcal{E}_{A}(w):=\frac{1}{2}\int_{\mathbb{R}^{n}}\,A\nabla{w}\cdot\nabla{w}\,dx-\int_{\mathbb{R}^{n}}\rho w 
\,dx,\\
&\mathcal{E}_{B}(w):=\frac{1}{2}\int_{\mathbb{R}^{n}}\,B\nabla{w}\cdot\nabla{w}\,dx-\int_{\mathbb{R}^{n}}\rho  w \,dx.
\end{split}
\end{equation*}
Hence $\mathcal{E}_{A}(w)\leq{}\mathcal{E}_{B}(w)$ for every $w\in D^1(\R^n)$. Since the solutions of~\eqref{div} are minimizers of these energies, compare with Lemma \ref{lm:ex},  we have  
$$
\mathcal{E}_{A}(u)=\min_{D^1(\R^n)}\,\mathcal{E}_{A}\leq{}\min_{D^1(\R^n)}\,\mathcal{E}_{B}=\mathcal{E}_{B}(v).
$$
Then
\begin{equation*}
-\frac{1}{2}\int_{\mathbb{R}^{n}}A\,\nabla{u}\cdot \nabla{u}\,dx=\mathcal{E}_{A}(u)\leq \mathcal{E}_{B}(v)=-\frac{1}{2}\int_{\mathbb{R}^{n}}B\,\nabla{v}\cdot \nabla{v}\,dx,
\end{equation*}
and 
\begin{equation*}
\begin{split}
-\frac{1}{2}\int_{\mathbb{R}^{n}}\,B\nabla{v}\cdot \nabla{v}\,dx &=\int_{\mathbb{R}^{n}}\,B\nabla{v}\cdot \nabla{v}\,dx-\int_{\R^{n}}\rho v\,dx
\\
&=\int_{\R^{n}}(B-A)\,\nabla v\cdot\nabla v\,dx+\int_{\R^{n}}A\,\nabla v\cdot \nabla v\,dx-\int_{\R^{n}}\rho v \,dx
\\
&\geq \int_{\R^{n}}(B-A)\,\nabla v\cdot \nabla v\,dx-\frac{1}{2}\int_{\mathbb{R}^{n}}A\,\nabla{u}\cdot \nabla{u}\,dx ,
\end{split}
\end{equation*}
concluding the proof.
\end{proof}

The following corollary  is an immediate consequence of the above lemma.

\begin{corollary}\label{cor:monG}
Let \(E\subset F\subset \R^n\) be two sets of finite measure. Then 
\[
\G_{\beta, K}(E)\ge \G_{\beta, K}(F)
\]
\end{corollary}

\begin{proof}
Let \((u_E,\rho_E)\) be the optimal pair for \(E\) and let \(v\) be a solution of 
\[
-\Div(a_F \nabla v)=\rho_E.
\]
Then \((v, \rho_E)\) is admissible in the minimization problem defining \(\G_{\beta, K}(F)\), hence
\[
\G_{\beta, K}(F)\le \int a_F |\nabla v|^2+K\int \rho_E^2\le \int a_E |\nabla u_E|^2+K\int \rho_E^2=\G_{\beta, K}(E),
\]
where the last inequality follow from Lemma~\ref{D_Lemma_1}.
\end{proof}

We conclude this section by proving the continuity of \(\G\) under \(L^1\) convergence. Recall that given two sets \(E\) and \(F\), \(E\Delta F:=(E\cup F)\setminus (E\cap F)\) is their symmetric difference.

\begin{proposition}\label{l:contG} 
Let \(\{E_h\}\) be a sequence of  sets with \(|E_h|=:V_h\to V>0\) and let \(E\) be such that \(|E_h\Delta E|\to 0\), so that in particular \(|E|=V\). Assume that \(\beta_h\to \beta\) and that \(K_h\to K\). 
 Then 
\[
\G_{\beta_h,K_h}(E_h)\to \G_{\beta,K}(E).
\]
Moreover,  \(\nabla u_{E_h}\) and \(\rho_{E_h}\) converge in \(L^2\) to \(\nabla u_E\) and \(\rho_E\) respectively.
\end{proposition}

\begin{proof}
Note that by \eqref{e:gbound}
\begin{equation}\label{e:supbound}
\sup_{h} \G_{\beta_h,K_h}(E_h)<+\infty.
\end{equation}
Thus 
\begin{equation*}
\sup_{h}\int_{\R^n} |\nabla v_{E_h}|^2+\int_{\R^n}\rho_{E_h}^2<\infty.
\end{equation*}
Moreover
\[
a_{E_{h}}=:a_h \overset{L^2}{\longrightarrow} a_E=\1_{E^c}+\beta \1_{E}.
\]
In particular, if \((u_h, \rho_h)=(u_{E_h}, \rho_{E_h})\) is the  minimizing pair for  \(\G_{\beta_h,K_h}(E_h)\), then up to subsequence there exists \((u, \rho)\) such that
\[
\nabla u_h\overset{L^2}{\rightharpoonup} \nabla u, \quad a_h\nabla u_h\overset{L^2}{\rightharpoonup} a_E\nabla u, \qquad \rho_h \overset{L^2}{\rightharpoonup}\rho.
\]
Since \((u_h, \rho_h)\) are in \(\A(E_h)\), one immediately deduces that \((u,\rho)\in \A(E)\) and thus, by lower semicontinuity,
\[
 \G_{\beta,K}(E)\le \int a_E |\nabla u|^2+K\int \rho^2\le \liminf_{h\to \infty} \int a_h |\nabla u_h|^2+K_h\int \rho_h^2.
\]
To prove the opposite inequality we take \((u_E,\rho_E)\) to be the minimizing pair for \(\G_{\beta,K}(E)\) and we define \((w_h, \widetilde \rho_h) \in \A(E_h)\) as 
\[
\widetilde \rho_h=\sigma_h\, \rho_E\1_{E_h},\qquad -\Div (a_h \nabla w_h)=\widetilde \rho_h,
\]
where \(\sigma_h\to 1\) is such that \(\int \widetilde \rho_h=1\). Since
\[
-\Div (a_h \nabla (u_E-w_h))=-\Div ((a_h-a_E) \nabla u_E)+\rho_E(\1_E-\sigma_h\1_{E_h}),
\]
by testing with \(u_E-w_h\) and by exploiting the Sobolev embedding  we obtain 
\[
\begin{split}
\|\nabla&(u_E-w_h)\|_{2}^2\le \int_{\R^n} a_h( \nabla u_E-\nabla w_h)\cdot (\nabla u_E-\nabla w_h)
\\
&{}={}\int (a_h-a_E) \nabla u_E\cdot\nabla (u_E-w_h)+\int \rho_E(\1_E-\sigma_h\1_{E_h})\rho_E (u_E-w_h)  
\\
&{}\le{} \|(a_h-a_E) \nabla u_E\|_{2}\|\nabla(u_E-w_h)\|_{2}+S_n \|(\rho_E(\1_E-\sigma_h\1_{E_h})\|_{2}\|\nabla(u_E-w_h)\|_{2}.
\end{split}
\]
Then Young's inequality implies that \(\|\nabla(u_E-w_h)\|_{2}\to 0\). 
Since also \(\|\widetilde\rho_h-\rho_E\|_{2}\to 0\) and \((w_h,\widetilde \rho_h)\in \A(E_h)\),  we get that
\[
\begin{split}
\limsup_{h\to \infty} \mathcal G_{\beta_h,K_h}(E_h)&\le \lim_{h\to \infty} \int_{\R^n} a_h|\nabla w_h|^2+K_h\int_{\R^n}\widetilde \rho_h^2
\\
&=\int_{\R^n} a_E|\nabla u_E|^2+K\int_{\R^n}\rho^2=\G_{\beta,K}(E).
\end{split}
\]
Strong convergence of \(\nabla u_{E_h}\) and \(\rho_{E_h}\) is now a simple consequence of the convergence of energies.
\end{proof}

\section{Small volume adjustments and Euler Lagrange equations}\label{s:EL}

In this section we show how to adjust the volume of a given set without increasing too much its energy which will be instrumental   both to prove compactness of the class of minimizers in Section~\ref{s:comp} and to get rid of the volume constraint in studying regularity of solutions of~\eqref{e:problem}, see Section~\ref{s:lambda}. The ``adjustment'' lemma will be proved with the aid of a deformation via a family of diffeomorphism close to the identity. Though not needed in the sequel we also establishes the  Euler Lagrange equations associated to~\eqref{e:problem}. We start with the following 
lemma.
\begin{lemma} \label{lemma1} 
For every \(\eta\in C_c^\infty(B_R;\R^n)\) there exists \(t_0=t_0(\dist(\spt \eta, \partial B_R)>0\) such that $\{\varphi_{t}\}_{|t|\leq t_{0}}$ defined by $\varphi_{t}(x):=x+t\,\eta(x)$ is a family of diffeomorphisms of \(B_R\) into itself. Moreover  for some set \(E\subset B_R\)  let \((u,\rho)\) be a solution of 
\[
-\Div(a_E \nabla u)=\rho.
\]
Then setting 
\[
u_{t}:=u\circ \varphi^{-1}_{t}\qquad\text{and}\qquad \widetilde{\rho_{t}}:=\det (\nabla\varphi_{t}^{-1})\,\rho\circ \varphi^{-1}_{t},
\]
 we have 
\begin{equation} \label{evvincolo}
-\Div\big(a_{E_{t}}\,A_{t}\nabla{u_{t}}\big)=\widetilde{\rho_{t}} 
\end{equation}
where  \(\|A_t-\Id\|_{\infty} =O(t)\) and  the implicit constant depends only on \(\|\nabla \eta\|_\infty\).
\end{lemma}
%
\begin{proof} The proof of the first part of the Lemma is straightforward. For the second we see that for \(\psi \in C_c^\infty\), by change of variables \(x=\varphi_t(y)\),
\begin{equation*}
\begin{split}
\int_{\mathbb{R}^{n}}\widetilde{\rho_{t}}(x)\,\psi(x)\,dx &
=\int_{\mathbb{R}^{n}}\rho(y)\psi (\varphi_{t}(y))\det(\nabla\varphi^{-1}_{t})(\varphi_t(y))\det(\nabla \varphi_{t} (y))\,dy
\\
&=\int_{\mathbb{R}^{n}}a_{E}(y)\nabla u(y)\cdot \nabla \left(\psi\circ\varphi_{t}\right)(y)\,dy
\\
&=\int_{\mathbb{R}^{n}}a_{E}(y)\nabla u(y)\cdot \big(\nabla{\varphi_{t}}(y)\big)^{T}  \nabla\psi(\varphi_{t}(y))\,dy
\\
&=\int_{\mathbb{R}^{n}}a_{E_{t}} \nabla{u}\circ \varphi_{t}^{-1} \left(\nabla{\varphi_{t}}\circ \varphi_{t}^{-1}\right)^{T}\nabla{\psi}\det \nabla \varphi^{-1}_{t}\,dx
\\
&=\int_{\mathbb{R}^{n}}a_{E_{t}} \big(\nabla{\varphi_{t}^{-1}}\big)^{-T}\nabla{u_{t}}\cdot \big(\nabla{\varphi_{t}}^{-1})^{-T}\nabla{\psi}\det \nabla{\varphi^{-1}_{t}}\,dx
\\
&=\int_{\mathbb{R}^{n}}a_{E_{t}}\,A_{t}\,\nabla{u_{t}}\cdot\nabla{\psi}\,dx.
\end{split}
\end{equation*}
Where we have used the equality  \(\nabla \varphi \circ \varphi_t^{-1}=(\nabla \varphi_t^{-1})^{-1}\) and for a matrix \(N\) we denoted by  \(N^T\)  its transpose and by \(N^{-T}\) for \((N^{-1})^T\). Hence $u_{t}$ is a solution of~\ref{evvincolo} with 
$$A_t=\det \nabla{\varphi^{-1}_{t}} \big(\nabla{\varphi_{t}^{-1}}\big)^{-T}\big(\nabla{\varphi_{t}^{-1}}\big)^{-1}.$$ By explicit computation  we see that \(A_t\) satisfies the desired bound.
\end{proof}

We now show how the energy \(\G\) changes by the effect of a family of diffeomprohism.
\begin{lemma} \label{G}
Let $E\subseteq B_M$ be a measurable set and let  $\{\varphi_{t}\}_{|t|\leq t_{0}}$ be a family of diffeomorphisms as in Lemma~\ref{lemma1}. Then
\begin{equation} \label{energytime}
\mathcal{G}_{\beta,K}\left(E_{t}\right)\leq{}\left(1+O(t)\right)\,\mathcal{G}_{\beta,K}(E)\mbox{,}
\end{equation}
where  $E_{t}:=\varphi_{t}\left(E\right)$ and  the implicit constant depends only on \(\|\nabla \eta\|_\infty\). Moreover
\begin{multline}\label{e:ELG}
\mathcal{G}_{\beta, K}\left(E_{t}\right)\le \mathcal{G}_{\beta,K}(E)
\\
+t\Bigg(\int_{\R^n}a_E \Big(|\nabla u_E|^2\Div \eta -2\nabla u_E\cdot \nabla \eta\, \nabla u_E\Big)-K\int_{\R^n} \rho_E^2 \Div \eta \Bigg)+O(t^2).
\end{multline}
\end{lemma}
\begin{proof} Let $(u_E,\rho_E)\in{}\mathcal{A}(E)$ be a the optimal pair for \(\G_{\beta,K}(E)\). By Lemma~\ref{lemma1} $u_{t}=u_E\circ{}\varphi^{-1}_{t}$ solves~\eqref{evvincolo} with \(\widetilde \rho_t=\rho_E\circ \varphi_t^{-1}\det (\nabla \varphi_t^{-1})\). 
Let $v_{t}$ be the solution of  
\begin{equation} \label{vincolo_v_{t}}
-\Div\,\big(a_{E_{t}}\nabla{v_{t}}\big)=\widetilde{\rho_{t}}\qquad\text{in}\quad\mathcal{D'}\big(\R^{n}\big)\mbox{.}
\end{equation}
\textsl{Step 1}: We start by proving  the following estimate
\begin{equation} \label{stima1}
\int_{\mathbb{R}^{n}}a_{E_{t}}\left(|\nabla{v_{t}}|^{2}-|\nabla{u_{t}}|^{2}\right)\,dx\leq O(t)\int_{\mathbb{R}^{n}}a_{E_{t}}|\nabla{u_{t}}|^{2}\,dx,
\end{equation}
where the implicit constant depends only on \(\|\nabla \eta\|_\infty\). In order to prove~\eqref{stima1} we claim that
\begin{equation} \label{eta1}
\left(\int_{\mathbb{R}^{n}}a_{E_{t}}|\nabla{(u_{t}-v_{t})}|^{2}\,dx\right)^{1/2}\le 
O(t)\left(\int_{\mathbb{R}^{n}}a_{E_{t}}|\nabla{u_{t}}|^{2}\right)^{1/2}\mbox{.}
\end{equation}
Indeed assuming that~\eqref{eta1} holds true and using that  $|a|^{2}-|b|^{2}=2b\cdot{}(a-b)+|a-b|^{2}$ for every $a$,$b\in\mathbb{R}^{n}$, we have 
\begin{equation} \label{formula5}
\begin{split}
\int_{\mathbb{R}^{n}}a_{E_{t}}\left(|\nabla{v_{t}}|^{2}-|\nabla{u_{t}}|^{2}\right)\,dx &=2\int_{\mathbb{R}^{n}}a_{E_{t}}\nabla{u_{t}}\cdot \nabla{(v_{t}-u_{t})}\,dx
\\ 
&+\int_{\mathbb{R}^{n}}a_{E_{t}}|\nabla{(u_{t}-v_{t})}|^{2}\,dx.
\end{split}
\end{equation}
We estimate the first term in the right hand side of~\eqref{formula5}. By~\eqref{eta1}, we find that
\begin{equation} \label{formula6}
\begin{split}
\int_{\mathbb{R}^{n}}a_{E_{t}}\,\nabla{u_{t}}\cdot{}\nabla{(v_{t}-u_{t})}\,dx
&\leq{}\left(\int_{\mathbb{R}^{n}}a_{E_{t}}|\nabla{u_{t}}|^{2}\right)^{1/2}\left(\int_{\mathbb{R}^{n}}a_{E_{t}}|\nabla{(u_{t}-v_{t})}|^{2}\right)^{1/2}\\
&\le O(t)\int_{\mathbb{R}^{n}}a_{E_{t}}|\nabla{u_{t}}|^{2}\,dx.
\end{split}
\end{equation}
By~\eqref{formula5} and~\eqref{formula6}, we have:
\begin{equation*} \label{formula7}
\begin{split}
\int_{\mathbb{R}^{n}}a_{E_{t}}\left(|\nabla{v_{t}}|^{2}-|\nabla{u_{t}}|^{2}\right)\,dx &\leq O(t)\int_{\mathbb{R}^{n}}a_{E_{t}}|\nabla{u_{t}}|^{2}\,dx+O(t^2)\int_{\mathbb{R}^{n}}a_{E_{t}}|\nabla{u_{t}}|^{2}\,dx,
\end{split}
\end{equation*}
which proves~\eqref{stima1}. 

Let us now  prove~\eqref{eta1}. By testing~\eqref{evvincolo} and~\eqref{vincolo_v_{t}} with $v_{t}-u_{t}$ we get
\begin{equation*} 
\begin{split}
&\int_{\mathbb{R}^{n}}a_{E_{t}}\,\nabla{v_{t}}\cdot{}\nabla{(v_{t}-u_{t})}\,dx=\int_{\mathbb{R}^{n}}\widetilde{\rho}_{t}\,(v_{t}-u_{t})\,dx=\int_{\mathbb{R}^{n}}a_{E_{t}}\,A_{t}\,\nabla{u_{t}}\cdot{}\nabla{(v_{t}-u_{t})}\,dx
\\
&=\int_{\mathbb{R}^{n}}a_{E_{t}}\,\left(A_{t}-\Id\right)\,\nabla{u_{t}}\cdot{}\nabla{(v_{t}-u_{t})}\,dx+\int_{\mathbb{R}^{n}}a_{E_{t}}\,\nabla{u_{t}}\cdot{}\nabla{(v_{t}-u_{t})}\,dx\mbox{.}
\end{split}
\end{equation*}
Rearranging terms and recalling that \(|A_t-\Id|=O(t)\), this  gives 
\[
\int_{\mathbb{R}^{n}}a_{E_{t}}|\nabla{(v_{t}-u_{t})}|^2\,dx\le O(t)\int_{\R^n} |\nabla v_t-\nabla u_t||\nabla u_t|,
\]
which, by Young's inequality, implies~\eqref{eta1}.

\medskip
\noindent
\textsl{Step 2}: 
By changing of variables   
\begin{equation*}
\int_{\mathbb{R}^{n}}a_{E_{t}}|\nabla{u_{t}}|^{2}\,dx=\int_{\R^n} |(\nabla \varphi_t)^{-T}\nabla u|^2\det \nabla \varphi_t.
\end{equation*}
Moreover
\[
\nabla \phi_t=\Id+t \nabla \eta +o(t)  \qquad \textrm{and}\qquad \det \nabla \phi_t=1+t\Div\eta+o(t),
\]
which gives 
\begin{equation} \label{stima2}
\int_{\mathbb{R}^{n}}a_{E_{t}}|\nabla{u_{t}}|^{2}\,dx=(1+O(t))\,\int_{\mathbb{R}^{n}}a_{E}|\nabla{u}|^{2}\,dx\mbox{.}
\end{equation}
and the more precise equality
\begin{multline}\label{e:El1}
\int_{\mathbb{R}^{n}}a_{E_{t}}|\nabla{u_{t}}|^{2}\,dx=\int_{\mathbb{R}^{n}}a_{E}|\nabla{u}|^{2}\,dx
\\
+t\int_{\R^n}a_E \Big(\Div \eta \,|\nabla u_E|^2-2\nabla u_E\cdot \nabla \eta\, \nabla u_E\Big)+o(t).
\end{multline} 
In the same way  we get
\begin{equation} \label{stima3}
\int_{E_{t}}\widetilde{\rho}_{t}^{2}\,dx=\int_E \frac{\rho^2}{\det \nabla \varphi_t} =(1+O(t))\,\int_{E}\rho^{2}\,dx\mbox{.}
\end{equation}
Furthermore, since \(\det \nabla \varphi_t=1+t\Div \eta +o(t)\), we also get 
\begin{equation}\label{e:El2}
\int_{E_{t}}\widetilde{\rho}_{t}^{2}\,dx=\int_{E}\rho^{2}\,dx -t \int_E \rho_E^2 \Div \eta+o(t).
\end{equation}

\medskip
\noindent
\textsl{Step 4:} Since, by its definition
\[
\int \widetilde \rho_t=1, \qquad \widetilde \rho_t\1_{E^{c}_t}=0,
\]
and \(v_t\) solves~\eqref{vincolo_v_{t}}, we see that \((v_t, \widetilde \rho_t)\in \A(E_t)\). Hence, by combining $~\eqref{stima1}$, $~\eqref{stima2}$ and $~\eqref{stima3}$ we obtain 
\begin{equation*}
\begin{split}
 \mathcal{G}_{\beta, K}(E_t)&\le\int a_{E_t} |\nabla v_t|^2+K\int \widetilde \rho_t^2 
\\
&\le (1+O(t))\int_{\R^n} a_{E_t} |\nabla u_t|^2+K\int \widetilde \rho_t^2
\\
& \leq (1+O(t))\Bigg(\int_{\mathbb{R}^{n}}a_{E}|\nabla{u_E}|^{2}\,dx+K\int_{E}\rho_E^{2}\,dx\Bigg)  =  (1+O(t))\mathcal G_{\beta,K} (E),
\end{split}
\end{equation*}
which proves~\eqref{energytime}. The proof of~\eqref{e:ELG} is obtained by combining the above argument with~\eqref{e:El1} and~\eqref{e:El2}.
\end{proof}

By combining  the Taylor expansion of the perimeter, \cite[Theorem 17.8]{Maggi12},
\[
P(\varphi_t(E))=P(E)+t \int_{\partial^* E} \Div_E \eta+o(t)\qquad \Div_E \eta=\Div \eta-\nu_E\cdot\nabla \eta\,\nu_E
\]
with~\eqref{e:ELG} we obtain the Euler Lagrange equations for minimizers of~\eqref{e:problem} whose proof is left to the reader.

\begin{corollary}
Let \(E\) be a minimizer of~\eqref{e:problem}, then 
\[
\int_{\partial^* E} \Div_E \eta+\int_{\R^n}a_E \Big(|\nabla u_E|^2\Div \eta -2\nabla u_E\cdot \nabla \eta\, \nabla u_E)-K\int_{\R^n} \rho_E^2 \Div \eta=0,
\]
for all \(\eta\in C_c(B_R;\R^n)\) with \(\int\Div \eta=0\).
\end{corollary}

The next series of results are  modeled after \cite{Almgren76} and  allow to do small volume adjustments without increasing too much the perimeter, see also \cite[Chapter 17]{Maggi12}.  The first lemma is elementary.

\begin{lemma} \label{c1} Let $E\subseteq\R^{n}$ be a set of finite perimeter and let  $U$ be an  open set such that $P(E,U)>0$. Then there exists $\varepsilon=\varepsilon(E)>0$, \(\gamma=\gamma(E)>0\) and a vector field \(\eta_E\in C^\infty_c(U;\R^n)\)  with \(\|\eta \|_{C^1}\le 1\) such that 
\begin{equation*}
\int_{E}\Div \eta_E \ge \gamma(E) >0.
\end{equation*}
\end{lemma}
\begin{proof} Since 
\[
P(E,U)=\sup \Biggl\{\int_{E} \Div \eta: \eta\in C_c^{1}(U;\R^n)\quad\|\eta\|_{\infty}\le 1\Biggr\},
\]
we find a vector filed such that 
\[
\int_{E} \Div \widetilde{\eta}\ge P(E,U)/2.
\]
Taking \(\eta=\widetilde \eta/\|\widetilde \eta\|_{C^1}\) we obtain the desired conclusion.
\end{proof}

In order to have uniform controls on the constants involved in our regularity theory, we need to enforce the above lemma in the following one, based on a concentration compactness argument. Note that this time the constants depend only on the upper bound on the perimeter, in particular they do not depend on \(R\).

\begin{lemma}\label{c3}
For every  \(P>0\) there exist    constants \(\bar \gamma=\bar\gamma(n,P)>0\) and \(\bar \delta=\bar \delta(n,P)\) such that if \(R\in (1,\infty)\) and \(E\subset B_R\) satisfies
\begin{equation}\label{e:1850}
\frac{|B_1|}{2}\le |E|\le \frac{3|B_1|}{2},\qquad P(E)\le P\,,
\end{equation}
 then  there exists a vector filed \(\eta\in C_c^1(B_R)\) with \(\| \eta\|_{C^1}\le 1\) such that 
\[
\int_{E} \Div \eta\ge \bar \gamma.
\]
\end{lemma}

\begin{proof}
 Let  us argue by contradiction: assume that there exist a sequence of radii \(R_k\) and a sequence of sets \(E_k\) satisfying \eqref{e:1850} such that 
\begin{equation}\label{e:contr}
\int_{E_k} \Div \eta \to 0\qquad \text{for all \(\eta\in C_c^1(B_{R-\bar \delta})\), with \(\| \eta\|_{C^1}\le 1\)},
\end{equation}
where  \(\bar\delta=\bar\delta (n,P)\) is  a small constant to be fixed later only in dependence of \(n\) and \(P\).  By \cite[Remark 29.11]{Maggi12} there exist points \(y_k\in \R^n\) and a constant \(\delta_1=\delta_1(n,P)\) such that
\[
|E_k\cap B_{1}(y_k)|\ge 2\delta_1.
\]
Then by taking \(z_k\in E_k\cap B_{1}(y_k)\subset B_{R_k}\cap B_{1}(y_k)\) we get 
\[
|E_k\cap B_{2}(z_k)|\ge 2\delta_1\qquad x_k\in B_{R_k}.
\]
Let us now detail the proof in the case in which, up to subsequences, \(R_k\to \infty\) and \(\partial B_{R_k-1}\cap B_2(z_k)\ne \emptyset\). The other cases are actually simpler and we explain how to modify the argument at the end of the proof. 
We first note that since  \(\partial B_{R_k-1}\cap B_2(z_k)\ne \emptyset\), we can take \(  x_k\in \partial B_{R_k}\) such that
\[
|E_k\cap B_{4}(x_k)|\ge |E_k\cap B_{2}(y_k)| \ge 2\delta_1\qquad x_k\in \partial B_{R_k}.
\]
Now a simple geometric argument ensures that
\[
\lim_{\delta \to 0} \sup_{k} |B_{4}(x_k)\cap (B_{R_k}\setminus B_{R_k-\delta})|\to 0.
\]
In particular we can chose \(\delta_2=\delta_2(n,P)\) such that
\begin{equation}\label{e:unpo}
|E_k\cap B_4(x_k)\cap B_{R_k-\delta_2}|\ge \delta_1.
\end{equation}
Let us now assume that, up to subsequences and a possible rotation of coordinates 
\[
F_k:=E_k\cap B_{4}(x_k)\cap B_{R_{k}-\delta_2}-x_k\to F\quad ,  \frac{x_k}{R_k}\to e_1
\]
where the first limit exits due to  our assumption on the perimeters. In particular  
\[
B_{4}\cap B_{R_{k}-\delta_2}(-x_k)\to \widehat B=B_{4}(0)\cap\{x_1<-\delta_2\}\,,
\]
and \(F\subset \widehat B\). Note that by \eqref{e:unpo},  \(F\ne \emptyset\) and, since \(|F_k|\le 3|B_1|/2\), \(| \widehat B\setminus F|>0\). In particular, \(P(F, \widehat B)>0\). By Lemma \ref{c1}, we can find a constant \(\gamma=\gamma_F>0\) and vector field \(\eta_F \in C^1_c(\hat B;R^d)\) with \(\| \eta \|_{C^1}\le 1\)  such that 
\[
\gamma\le \int _{F}\Div \eta_F. 
\]
For \(k\) large,  the vector field \(\eta_k=\eta_F(\cdot+x_k)\) satisfies,  \(\eta_k\in C_c(B_{R_k-\delta_2/2})\), \(\| \eta \|_{C^1}\le 1\) and  contradicts \eqref{e:contr} with \(\bar \delta=\delta_2/2\).

Let us conclude by explaining how to modify the proof in the case in which either  \(B_{2}(z_k)\cap \partial B_{R_k-1}=\emptyset\) or \(R_k\to \bar R<\infty\).  In the  first  case instead one argue as above by considering  the set \(F_k=E_k\cap B_{2}(z_k)-z_k\) and by noticing that the  vector fields \(\eta_k= \eta_F(\cdot+y_k)\), \(F=\lim F_k\), are  compactly supported in \(B_{R_k-1/2}\). In the second case one can simply reproduce the above argument.
\end{proof}
%
The next proposition will be crucial in removing the volume constraint and in making comparison estimates for minimizers of \eqref{e:problem}. The proof is based on a concentration-compactness argument.

\begin{proposition} \label{c2}
For every  \(P>0\) there exist    constants \(\bar\sigma=\bar\sigma(n,P)>0\) and \( C=C(n)\) such that if \(R\in (1,\infty)\) and \(E\subset B_R\) satisfies
\begin{equation*}
\frac{|B_1|}{2}\le |E|\le \frac{3|B_1|}{2},\qquad P(E)\le P\,,
\end{equation*}
then  for all $\sigma\in{}(-\bar\sigma,\bar\sigma)$ there exists  $F_{\sigma}\subset B_{ R}$ such that 
\[
|F_{\sigma}|=|E|+\sigma\qquad{\rm and}\qquad|\F_{\beta,K,Q}(F_\sigma)-\F_{\beta,K,Q}(E)|\leq{}C|\sigma|\,\F_{\beta,K,Q}(E).
\]
\end{proposition}
\begin{proof} By Lemma~\ref{c3} we can find \(\bar \gamma=\bar \gamma(n,P)>0\), \(\bar \delta=\bar\delta(n,P)\) and a vector field  $\eta\in{}C^{1}_{c}\big(B_{R-\bar \delta};\R^{n}\big)$ with \(\|\eta\|_{C^1}\le 1\) such that 
\begin{equation}\label{e:sale}
\bar \gamma\le \int_{E}\Div \eta.
\end{equation}
Define a family of diffeomorphisms $\varphi_{t}:=\Id+t\,\eta$ and note that, since \(\dist(\spt(\eta,\partial B_{R})\ge \bar \delta(n,P)\), they send \(B_{R}\) into itself for \(|t|\le t_0(n,P)\) . By Taylor expansion 
\begin{equation}\label{e:vol}
|E_{t}|=|E|+t \int_{E}\Div \eta+O(t^{2})|E|\mbox{.}
\end{equation}
and
\begin{equation*}\label{e:per} 
 P(E_{t})=P(E)+t\,\int_{\partial^{*}G}\Div_{E}\eta\,d\mathcal{H}^{n-1}+O(t^{2})P(E),
\end{equation*}
where the implicit constants depends only on \(\|\nabla \eta\|_{\infty}\le 1\). Moreover
\begin{equation}\label{e:Gv}
\G_{\beta,K}(E_t)\le (1+C|t|) \G_{\beta,K}(E),
\end{equation}
where $E_{t}=\varphi_{t}(E)$ and the constant in~\eqref{e:Gv}  depends only on \(\|\nabla \eta\|_{\infty}\le 1\). Hence  we can find  $t_{1}=t_1(n,P)>0$ such that  
\begin{subequations}
\begin{equation} \label{e2}
\big||E_{t}|-|E|\big|\geq{}|t|\,\frac{\bar\gamma}{2} \qquad\text{(by \eqref{e:sale})},
\end{equation}
\textrm{and}
\begin{equation} \label{e3}
|\F_{\beta,K,Q}(E_t)-\F_{\beta,K,Q}(E)|\leq{}C|t|\F_{\beta,K,Q}(E).
\end{equation}
\end{subequations}
for every $|t|\le t_1$. By equations~\eqref{e2} and~\eqref{e3} we get
\[
|\F_{\beta,K,Q}(E_t)-\F_{\beta,K,Q}(E)|\le C \F_{\beta,K,Q}(E)\bigl||E_{t}|-|E|\bigr|.
\]
Let  $g(t):=|E_{t}|$ and note that thanks to~\eqref{e:vol} and \eqref{e:sale}, \(g\) is increasing in a neighborhood of \(0\). Take $\bar\sigma>0$ such that $(|E|-\bar\sigma,|E|+\bar \sigma)\subseteq{}g\big((-t_{1},t_{1})\big)$. Then for every $|\sigma|\leq{}\bar \sigma$ there exists $t_{\sigma}>0$ such that $|E_{t_{\sigma}}|=|E|+\sigma$. Setting \(F_\sigma=E_{t_\sigma}\) we obtain the desired conclusion.
\end{proof}

\section {$\Lambda$-minimality and local variations}\label{s:lambda}
In order to study the regularity of minimizers it will be convenient to understand what is the behavior under small perturbations in balls. In this section we start by removing the volume constraint by showing that minizers are \(\Lambda\)-minimizer of \(\F\) under small perturbations. In order to keep track of the dependence of the parameters in Theorem~\ref{thm:maineps}, it will be important that this ``almost''-minimality depends only on the structural parameter of the problem. We start thus by fixing the following convention, which will be in force throughout all the rest of the paper:

\begin{Convention}[Universal constants]\label{c:universal constants}
Given \(A>0\), we say that  \(\beta, K, Q\) with \(\beta\ge 1\) are \emph{controlled by} \(A\) if 
\[
\beta+K+\frac{1}{K}+Q\le A.
\]
We will also say that a constant is universal if  it depends only on the dimension \(n\) and on \(A\). Finally, for two positive quantities \(X\) and \(Y\),  we will sometimes write \(X\lesssim Y\) if there exists a universal constant \(C\) such that \(X\le C Y\) and we write \(X\gtrsim Y\) if \(Y\lesssim X\).
\end{Convention}
Note in particular that universal constants \emph{do not depend} on the size of the container where the minimization problem is solved. Moreover we also remark here the following elementary fact: since \(B_1\) is always a competitor for \eqref{e:problem},  if \(E\) is a minimizer then
\begin{equation}
\label{e:pfbound}
P(E)\le \mathcal F_{\beta,K,Q}(E)\le  \mathcal F_{\beta,K,Q}(B_1)\le C(n,A),
\end{equation}
whenever  \(\beta, K, Q\) are controlled by \(A\).

Let us now introduce the following perturbed minimality condition.

\begin{definition} [\((\Lambda, \bar r)\)-minimizer] We say that $E$ is a $( \Lambda, \bar r)$-minimizer  of the energy $\mathcal{F}$ if there exist constants \(\Lambda>0\) and  $\bar r>0$ such that for every ball $B_{r}(x)\subseteq\mathbb{R}^{n}$ with $ r\leq \bar r$ we have
\begin{equation}\label{e:Lambdaprob}
\mathcal{F}_{\beta,K,Q}(E)\leq \mathcal{F}_{\beta,K,Q}(F)+ \Lambda \,|E\Delta F|\qquad \text{whenever $E\Delta F\subset B_{\bar r}(x)$.}
\end{equation}
\end{definition}

\begin{remark}\label{rmk:radius}
Note that if \(E\) is \((\bar \Lambda, \bar r)\)-minimizer than it is also a \((\bar \Lambda_1, \bar r_1)\)-minimizer whenever  \(\bar \Lambda_1\ge\bar \Lambda\) and \(\bar r_1\le \bar r\). Hence there is no loss of generality in assuming  that \(\bar r\le 1\). 
\end{remark}

We can now establish the desired \(\Lambda\)-minimality property for minimizers of \eqref{e:problem}.
\begin{proposition} \label{t2}
Let \(A>0\) and let    \(\beta, K, Q\) with \(\beta\ge 1\) be  controlled by \(A\) and let \(R\ge  1\).
Then there exist \( \Lambda_1, \bar r_1>0\) universal  such that  all minimizers~\eqref{e:problem}  satisfy
\begin{equation*}
\mathcal{F}_{\beta,K,Q}(E)\leq \mathcal{F}_{\beta,K,Q}(F)+\Lambda_1 \,|E\Delta F| ,
\end{equation*}
whenever  \(F\subset B_R\) and $E\Delta F\subset B_{r}(x_{0})$, \(r\le \bar r_1\).
\end{proposition}
\begin{proof}  
 Clearly we can suppose  that 
\[
\mathcal F_{\beta,K, Q}(F)\le \mathcal F_{\beta,K, Q}(E)\lesssim 1\,,
\]
since otherwise the result is trivial. In particular \(P(F)\) is bounded by an universal constant \(P\).  Let \(\bar \sigma\) and \(C\) be the parameters  in Proposition~\ref{c2} associated to \(P\). If \(\bar r_1\) is chosen small enough we have
\[
|E\Delta F|\le \omega_n\bar r_1^n\ll \bar \sigma.
\]
Moreover, since \(|E|=|B_1|\), \(|F|\in (|B_1|/2,3|B_1|/2)\). Hence we can apply Proposition~\ref{c2} to \(F\) to obtain a set \(\widetilde F\subset B_{ R}\) such that \(|\widetilde F|=|B_1|\) and
\begin{equation}\label{e:ql}
\F_{ \beta,K, Q}(E) \le \F_{ \beta, K,Q}(\widetilde F)\leq\big(1+C\bigl|| \widetilde F|-|F|\bigr|\big)\,\F_{\beta,K,Q} (F),
\end{equation}
where the first inequality is due to the minimality of \(E\). Since  \(\mathcal F_{\beta,K, Q}(F)\lesssim 1\) and 
\[
\bigl||\widetilde F|-|F|\bigr|=||E|-|F||\le |F\Delta E|,
\]
we obtain the conclusion for a suitable universal constant \(\Lambda_1\).

\end{proof}
 
We conclude this section by establishing the following ``local'' minimality properties of  minimizers~\eqref{e:problem}. Note that in (ii) below we are not requiring \(F\) to be contained in \(B_R\).

\begin{proposition} \label{in_out}
Let \(A>0\), and let  \(\beta, K, Q\) be  controlled by \(A\)  and \(R\ge 1\).  Then there exist universal constants \(\Lambda_2\) and \(\bar r_2\) such that  all minimizers~\eqref{e:problem}
 satisfy the following two properties:
\begin{enumerate} [\upshape (i)]
\item for every set of finite perimeter $F\subseteq{}E$ with $E\setminus{}F\subset{}B_{r}(x)$ and $r\leq{}\bar r_2$ it holds:
\begin{equation} \label{FinE}
P(E)\leq{} P(F)+\Lambda_2 |E\setminus{}F|+\Lambda_2 Q^2 \int_{E\setminus{}F}|\nabla{}u_E|^{2}\,dx.
    \end{equation}
\item for every set of finite perimeter $F\supseteq{}E$ with $F\setminus{}E\subset{}B_{r}(x)$ and $r\leq{} \bar r_2''$ it holds:
                 \begin{equation} \label{FoE}
P(E)\leq{}P(F)+\Lambda_2|F\setminus{}E|\mbox{.}
                 \end{equation}
\end{enumerate}
In particular,
\begin{equation}\label{e:precordes}
P(E)\leq{}P(F)+\Lambda_2 |E\Delta{}F|+\Lambda_2 Q^2\int_{E\Delta{}F}|\nabla{}u|^{2}\,dx\mbox{,}
\end{equation}
whenever $F\Delta{}E\subset{}B_{r}(x)$ with $r\leq{}\bar r_2 $.
\end{proposition}
\begin{proof} 
We start proving (i). Let  \(E\) be a minimizer and \((u_E,\rho_E)\) be the minimizing pair for \(\G(E)\). Let $F\subseteq{}E$ be such that $E\setminus{}F\subset{}B_{r}(x)$ with $r\leq{}\bar r_1 $ where \(\bar r_1\) is the constant defined in Proposition~\ref{t2}, by possibly choosing \(r_1\) smaller,   we can assume that
\begin{equation}\label{volF}
|F|\ge \frac{|E|}{2}=\frac{|B_1|}{2}.
\end{equation}
Let us set 
$$\rho=(\rho_E+\lambda_{F})\,\1_{F}\qquad \text{ where }\quad \lambda_{F}=\frac{\int_{E\setminus{}F}\rho_E\,dx}{|F|},$$
and let \(u\) be the solution of 
\[
-\Div (a_F \nabla u)=\rho.
\]
Note that \((u,\rho)\in \A(F)\) and thus, by using the \(\Lambda\)-minimality of \(E\) established in Proposition~\ref{t2},
\[
\begin{split}
P(E)+Q^2\Bigl(\int_{\R^n}a_E |\nabla u_E|^2\,dx+K\int \rho_E^2\,dx\Bigr) \le P(F) +Q^2\Bigl(\int_{\R^n} a_F |\nabla u|^2\,dx+& K \int \rho^2\,dx\Bigr) \\&+\Lambda |E\setminus F|.
\end{split}
\]
Item  (ii) will then  follow if we can prove
\begin{equation}\label{e:e1}
\int \rho^2-\rho_E^2\lesssim|E\setminus F|,
\end{equation}
and 
\begin{equation}\label{e:e2}
\int_{\mathbb{R}^{n}}a_{F}\,|\nabla{}u_{F}|^{2}-a_{E}|\nabla{}u|^{2}\,dx \lesssim |E\setminus F|+ \int_{E\setminus{}F}|\nabla{}u_E|^{2}.
\end{equation}
 To prove~\eqref{e:e1} we estimate 
\begin{equation*}
\begin{split}
\int_{\mathbb{R}^{n}}(\rho^{2}-\rho_E^{2})\,dx &=-\int_{E\setminus{}F}\rho_E^{2}\,dx+\int_{F}(\lambda_{F}^{2}+2\rho_E\,\lambda_{F})\,dx
\\
&\leq - \int_{E\setminus F} \rho_E^2 \, dx + \frac{|E\setminus F|}{|F|} \int_{E\setminus F} \rho_E^2 + 2 \|\rho_E\|_\infty \int_{E\setminus F} \rho_E \, dx 
\\
&\le 2 \|\rho_E\|^2_{\infty} |E\setminus F|,
\end{split}
\end{equation*}
where in the first inequality we have used~\eqref{volF} and the definition of \(\lambda_F\). By \eqref{e:rhobound},  \(\|\rho_E\|_{\infty}\lesssim 1\) and this concludes the proof of~\eqref{e:e1}.

Let us now prove~\eqref{e:e2}.  First note that 
\begin{equation} \label{e_3}
\begin{split}
\int_{\mathbb{R}^{n}}a_{F}|\nabla{}u|^{2}-a_{E}|\nabla{}u_E|^{2}\,dx &=\int_{\mathbb{R}^{n}}a_{F}(\,|\nabla{}u|^{2}-|\nabla{}u_E|^{2})\,dx\\
&+\int_{\mathbb{R}^{n}}(a_{F}-a_{E})\,|\nabla{}u_E|^{2}\,dx.
\end{split}
\end{equation}
Testing the equations satisfied by \(u_E\) and \(u\) with \(u_E\) and \(u\) respectively and subtracting the result we obtain also
\begin{equation}\label{e:fk}
\int_{\mathbb{R}^{n}}a_F|\nabla{}u|^{2}-a_{E}|\nabla{}u_E|^{2}\,dx=\int_{R^n} u\rho\,dx- \int_{\R^n} u_E\rho_E\,dx.
\end{equation}
Subtracting~\eqref{e_3} from two times~\eqref{e:fk} we get
\begin{equation}\label{e:fk1}
\begin{split}
\int_{\mathbb{R}^{n}}a_{F}|\nabla{}u|^{2}-a_{E}|\nabla{}u_E|^{2}\,dx &=\int_{\mathbb{R}^{n}}a_{F}\left(\,|\nabla{}u_E|^{2}-|\nabla{}u|^{2}\right)\,dx
\\
&\quad+\int_{\mathbb{R}^{n}}(a_{E}-a_{F})\,|\nabla{}u_{E}|^{2}\,dx
\\
&\quad+2\int_{\R^n} u\rho\,dx- 2\int_{\R^n} u_E\rho_E\,dx.
\end{split}
\end{equation}
Moreover,
\begin{equation} \label{e_3'}
\begin{split}
\int_{\mathbb{R}^{n}}a_{F}\left(\,|\nabla{}u_{E}|^{2}-|\nabla{}u|^{2}\right)\,dx &=2\int_{\mathbb{R}^{n}}a_{F}\,\nabla{}u\cdot{}(\nabla{}u_E-\nabla{}u)\,dx
\\
&\quad+\int_{\mathbb{R}^{n}}a_{F}\,|\nabla{}u_{E}-\nabla{}u|^{2}\,dx
\\
&=2\int_{\mathbb{R}^{n}}\rho(u_E-u)\,dx
\\
&\quad+\int_{\mathbb{R}^{n}}a_{F}|\nabla{}u_{F}-\nabla{}u|^{2}\,dx\mbox{.}
\end{split}
\end{equation}
Combining \eqref{e:fk1} and \eqref{e_3'} we then obtain:
\begin{equation} \label{e_4}
\begin{split}
\int_{\mathbb{R}^{n}}a_{F}|\nabla{}u|^{2}-a_{E}|\nabla{}u_E|^{2}\,dx &=2\int_{\mathbb{R}^{n}}(\rho-\rho_E)\,u_E\,dx+\int_{\mathbb{R}^{n}}a_{F}\,|\nabla{}u-\nabla{}u_E|^{2}\,dx
\\
&\quad+\int_{\mathbb{R}^{n}}(a_{E}-a_{F})\,|\nabla{}u_E|^{2}\,dx\mbox{.}
\end{split}
\end{equation}
We start to estimate the first term in the right hand side of~\eqref{e_4}. By using Proposition~\ref{firstreg} and by arguing as in the proof of \eqref{e:e1} the first term can be easily estimated as
\begin{equation*}
\int_{\mathbb{R}^{n}}(\rho-\rho_E)\,u_E\lesssim |E\setminus F|.
\end{equation*}
To estimate the second term in the right hand side of~\eqref{e_4}, we write  
$$-\Div \big(a_{F}(\nabla{}u-\nabla u_E)\big)=\rho-\rho_E+\Div\big((a_{F}-a_{E})\,\nabla{}u_E\big)\mbox{.}$$ 
Therefore
\begin{equation*} \label{i2}
\begin{split}
\int_{\mathbb{R}^{n}}a_{F}\,\left|\nabla{}u-\nabla{}u_E\right|^{2}\,dx &=\int_{\mathbb{R}^{n}}(\rho-\rho_E)\,(u-u_E)\,dx\\
&+\int_{\mathbb{R}^{n}}(a_{E}-a_{F})\,\nabla{}u_E\cdot{}\left(\nabla{}u-\nabla{}u_E\right)\,dx
\\
&\leq{}\|\rho-\rho_E\|_{(2^{*})'}\,\|u-u_E\|_{2^{*}}
\\
&+\left(\int_{\mathbb{R}^{n}}(a_{F}-a_E)^2\,\left|\nabla{}u_E\right|^{2}\right)^{\frac{1}{2}}\,\|\nabla u-\nabla u_E\|_{2}.
\end{split}
\end{equation*}
By  the Sobolev embedding and Young inequality (and recalling that \(1\le a_F\le \beta\)),  the above inequality immediately imply
\[
\int_{\mathbb{R}^{n}}a_{F}\,\left|\nabla{}u-\nabla{}u_E\right|^{2}\lesssim \int_{\mathbb{R}^{n}}(a_{F}-a_E)^2\,\left|\nabla{}u_E\right|^{2}+ \|\rho-\rho_E\|_{(2^{*})'}^2.
\]
By the definition of \(\rho\), the second term  is \(\lesssim |E \setminus F|\) (note that \(2/(2^{*})'\ge 1\))  while the first one is less than
\[
\beta^2\int_{E\setminus F} |\nabla u_E|^2.
\]
Since also the third term in~\eqref{e_4} can be estimated by the above integral, this concludes the proof of \eqref{e:e2}.

Let us now prove (ii). Let \(F\supseteq E\), note that \(P(F\cap B_R)\le P(F)\) and that \((F\cap B_R)\setminus E\subset F\setminus E\). Hence if we can prove (i) for subsets of \(B_R\) we will get it for all sets. Let us then  assume that \(E\subseteq F\subseteq B_R\). By \(\Lambda\)-minimality of \(E\)
\[
\F_{\beta,K, Q}(E)\le \F_{\beta,K, Q}(F)+\Lambda |F\setminus E|.
\]
Since, by Lemma~\ref{D_Lemma_1}, \(\G_{\beta,K}(E)\ge \G_{\beta,K}(F)\) the conclusion follows.
\end{proof}

\begin{remark}\label{rmk:cordes}
We record here the following simple consequence of \eqref{e:precordes}. Assume that \(|\nabla u_E|^2\in L^{p}\), then \eqref{e:precordes} and H\"older inequality imply that for \(F\) such that $F\Delta{}E\subset{}B_{r}(x)$ with $r\leq{}\bar r_2 $, 
\begin{equation*}
\begin{split}
P(E)&\leq{}P(F)+\Lambda_2 |E\Delta{}F|+\Lambda_2 Q^2\int_{E\Delta{}F}|\nabla{}u|^{2}\,dx
\\
&\le P(F)+\Lambda_2|B_r|+\Lambda_2 Q^2 |B_r|^{1-\frac{1}{p}}\|\nabla u_E\|_{2p}^2\le P(F)+Cr^{n-\frac{n}{p}}.
\end{split}
\end{equation*}
In particular if \(p>n\), then \(n-\frac{n}{p}>n-1\) and thus \(E\) is a \(\omega\) minimizers of the perimeter in the sense of \cite{Tamanini82}. Hence \(\partial E\) is a \(C^1\) manifold outside a singular  closed set \(\Sigma\) of dimension at most \((n-8)\). Note that by Cordes estimate,~\cite{Cordes56}, the assumption \(|\nabla u_E|^2\in L^{p}\) with \(p>n\) is satisfied wherever \(\beta-1\ll 1\). In particular, in this regime,   Taylor cones singularities are  excluded in \(\mathbb R^3\).
\end{remark}

\section{Compactness of minimizers}\label{s:comp}
In this section we prove  that the class of minimizers of~\eqref{e:problem} is a compact subset of \(L^1\), this is not really necessary in the proof of the main result, but we believe it can be interesting by its own.
\begin{proposition}\label{prop:compactenss}
Let  \(K_h, Q_h\in \R\), \(\beta_h\ge 1\) and \(R_h\ge 1\) be such that 
\[
K_h\to K>0\,,\quad \beta_h\to \beta\ge 1\,,\quad R_h\to R\ge 1\,,\quad Q_h\to Q\ge 0.
\]
For every  \(h\in \mathbb N\) let \(E_h\) be  a minimizer of~\eqref{e:problem} (with \(\beta, K, Q, R\) replaced by \(\beta_h, K_h, Q_h, R_h\), then, up to a non relabelled subsequence,  there exists a set of finite perimeter \(E\) such that
\begin{equation}\label{e:conv}
|E\Delta E_h|\to 0.
\end{equation}
Moreover \(E\) is a minimizer of~\eqref{e:problem} and 
\[
 \F_{\beta,K, Q}(E)=\lim_{h\to \infty} \F_{\beta_h,K_h, Q_h}(E_h)\qquad P(E_h)\to P(E).
\]
\end{proposition}

\begin{proof}
Since if \(R_h=1\) for all \(h\) the problem is trivial (recall that \(|E_h|=|B_1|\)) we can assume that \(R_h\) and \(R\) are strictly bigger than one. Moreover \(B_1\) is always an admissible competitor and thus,
\[
\limsup_{h} \F_{\beta_h,K_h, Q_h}(E_h)\le \limsup_{h} \F_{\beta_h,K_h, Q_h}(B_1)=C(n,K,Q, \beta).
\]
In particular the perimeters of \(E_h\) are  uniformly bounded and since all the sets are included in, say, \(B_{2R}\) there exists a non relabelled subsequence and  set \(E\subset B_R\) such that~\eqref{e:conv} hold true. Since the perimeter is lower-semicontinuous and, by Proposition~\ref{l:contG}, \(\G\) is continuous we also get that
\begin{equation}\label{e:lsc}
 \F_{\beta,K, Q}(E)\le \liminf_{h} \F_{\beta_h,K_h, Q_h}(E_h).
\end{equation}
We now show that \(E\) is a minimizer. For let  \(F\subset B_R\) with  \(|F|=|B_1|\). Since \(R_h\to R\), we  can find \(\lambda_h\to 1\) such that  \(F_h:=\lambda_h F\subset B_{R_h}\). Clearly, \(F_h\to F\),   \(|F_h|=|F|+o(1)\) and  \(P(F_h)=P(F)+o(1)\). Thus  
\begin{equation}\label{e:quasi}
 \F_{\beta,K, Q}(F)= \F_{\beta_h,K_h, Q_h}(F_h)+o(1).
\end{equation}
 By Proposition~\ref{c2} applied to  \(F_h\)  we can find sets  \(\widetilde{F}_h\subset B_{R_h}\) such that \(|\widetilde F_h|=|B_1|\) and
\[
 \F_{\beta_h,K_h, Q_h}(\widetilde{F}_h)= \F_{\beta_h,K_h, Q_h}(F_h)+o(1)= \F_{\beta,K, Q}(F)+o(1).
\] 
where in the last equality we have used~\eqref{e:quasi}. By minimality of \(E_h\) we get
\[
\F_{\beta_h,K_h, Q_h}(E_h)\le \F_{\beta_h,K_h, Q_h}(\widetilde{F}_h)= \F_{\beta,K, Q}(F)+o(1),
\]
which combined with~\eqref{e:lsc}  implies the minimality of \(E\). By choosing \(E=F\) we also deduce the convergence of the energies and, by Proposition ~\ref{l:contG}, this implies the convergence of the perimeters.
\end{proof}

\section {Decay of the Dirichlet energy and density estimates}\label{s:Dir}
\subsection{Decay of the Dirichlet energy}\label{ss:dir} Following \cite{FuscoJulin15}, in this subsection we establish an  almost Lipschitz decay for the   Dirichlet energy of \(u_E\) in certain regimes. Namely when the set  or the complement almost fill a ball or when the set is very close to an half space. 

We start by recalling the following higher  integrability lemma for solution of~\eqref{vincolo1}  The proof can be found for instance in~\cite{Giusti03}.

\begin{lemma}\label{l:higherintegrability}
Let \(E\) be set of finite measure and let \((u,\rho)\in \A(E)\). Then there exists \(C=C(n,\beta)\) and \(p=p(n,\beta)>1\) such that for all balls \(B_r(x)\subset \R^n\)
\begin{equation} \label{G-M_2}
\fint_{B_{r}(x)}|\nabla{u}|^{2p}\,dx\leq C\,\left\{\left(\fint_{B_{2r}(x)}|\nabla u|^{2}\,dx\right)^{p}+r^{2p}\,\fint_{B_{2r}(x)}\rho^{2p}\,dx \right\}.
\end{equation}
Furthermore, the constants \(C\) and \(p\) depends only on an upper  bound for \(\beta\).
\end{lemma}

We start with the following elementary lemma where the optimal decay is obtained in some  limit situations.

\begin{lemma}\label{lm:exact_decay}
Let \(\beta \ge 1\) and  \(\rho \in L^\infty(\R^n) \). Then  here exists a dimensional constant \(C=C(n)\) such that:
\begin{enumerate} [\upshape (i)]
\item
 if   \(v\in W^{1,2}(B_r(x))\) is a solution of 
\[
-\Delta v=\rho , 
\]
then for all \(\lambda \in (0,1)\)
\begin{equation}\label{e:decA}
\fint_{B_{\lambda r}(x)}|\nabla v|^{2}\,dx\leq C\,\fint_{B_{r}(x)}|\nabla v|^{2}\,dx+\frac {C}{\lambda^n}r^{2}\,\|\rho\|_{\infty}^{2}.
\end{equation}
\item
 If  \(v\in W^{1,2}(B_r(x))\) is a solution of 
\[
-\Div(a_H \nabla v) =\rho, \qquad a_H=\beta \1_{H}+\1_{H^c},
\]
where \(H:=\bigl\{y\in\R^n: (y-x)\cdot e\le 0\bigr\}\) for some  \(e\in \mathbb S^{n-1}\).  Then for all \(\lambda \in (0,1)\)
\begin{equation}\label{e:decH}
\fint_{B_{\lambda r}(x)}a_H|\nabla v|^{2}\,dx\leq C\,\fint_{B_{r}(x)}a_H |\nabla v|^{2}\,dx+\frac {C}{\lambda^n}r^{2}\,\|\rho\|_{\infty}^{2}.
\end{equation}
\end{enumerate}

\end{lemma}

\begin{proof}
We just prove point (ii) since (i) is a particular case (and well known). By scaling  and translating, we can assume without loss of generality that \(x=0\) and \(r=1\). Let \(w\) be the solution of 
\[
\begin{cases}
-\Div (a_H \nabla w)=0\qquad &\text{in \(B_1\)}\\
w=v& \text{on \(\partial B_1\),}
\end{cases}
\]
so that \(u=v-w\) solves 
\[
\begin{cases}
-\Div (a_H \nabla u)=\rho \qquad&\text{in \(B_1\)}\\
u=0& \text{on \(\partial B_1\).}
\end{cases}
\]
By multiplying the last equation by \(u\),  applying Poincar\'e inequality we obtain
\[
\int_{B_1}a_H|\nabla u|^2\le  \|\rho\|_{2}\|u\|_{2}\le  C(n)\|\rho\|_{\infty }\|\nabla u\|_{2} \le C(n)\|\rho\|_{\infty } \left (\int_{B_1}a_H|\nabla u|^2\right)^{\frac{1}{2}}.
\]
where we have used that \(a_H \ge 1\). Hence 
\begin{equation}\label{e:vw}
\int_{B_1} a_H |\nabla v-\nabla w|^2=\int_{B_1}a_H |\nabla u|^2\le C \|\rho\|^2_{\infty}.
\end{equation}
Moreover, by \cite[Lemma 2.3]{FuscoJulin15},
\[
\fint_{B_\lambda} a_H|\nabla w|^2\le \fint_{B_1} a_H|\nabla w|^2.
\]
Hence,
\begin{equation}\label{e:una_sola}
\begin{split}
\fint_{B_\lambda} a_H|\nabla v|^2&\le 2\fint_{B_\lambda } a_H |\nabla w|^2+2\fint_{B_\lambda} a_H|\nabla v-\nabla w|^2
\\
&\le 2\fint_{B_1} a_H|\nabla w|^2+2\fint_{B_\lambda} a_H|\nabla v-\nabla w|^2
\\
&\le 4\fint_{B_1} a_H|\nabla v|^2+2\fint_{B_\lambda} a_H|\nabla v-\nabla w|^2+4\fint_{B_1} a_H |\nabla v-\nabla w|^2.
\end{split}
\end{equation}
which together with \eqref{e:vw} concludes the proof.
\end{proof}

As in \cite{FuscoJulin15},  we now exploit the higher integrability of \(\nabla u_E\) recalled  in Lemma \ref{l:higherintegrability} to obtain an ``almost version'' of the above decay.  

\begin{proposition} [Decay of Dirichlet energy] \label{1decay}
Let  \( \beta \ge 1\) then there exists a constant \(C=C(n, \beta )\) with the following property: if \(E\subset \R^n\), \(u\) and \(\rho\) satisfy
\[
-\Div (a_E \nabla u)=\rho, \qquad a_E=\beta \1_{E}+1_{E^c},
\]
then  for  all  $\lambda\in{}\left(0,\frac{1}{2}\right)$ there exists $\varepsilon_{0}=\varepsilon_{0}(\lambda, \beta)>0$ such that
\begin{enumerate}[\upshape (i)]
\item if  
\[
\text{either} \qquad 
\frac{|E\cap{}B_{r}(x)|}{|B_{r}(x)|}\leq{}\varepsilon_{0} \qquad  \text{or}  \qquad 
\frac{|B_{r}(x)\setminus{}E|}{|B_{r}(x)|}\leq{}\varepsilon_{0},
\]
then
\begin{equation*}
\fint_{B_{\lambda r}(x)}|\nabla{u_E}|^{2}\,dx\leq{}C\,\fint_{B_{r}(x)}|\nabla{u_E}|^{2}\,dx+\frac {C r^{2}}{\lambda^n}\|\rho\|^2_{\infty}\mbox{.}
\end{equation*}
\item If 
 \[
 \frac{|(E\Delta H)\cap B_{r}(x)|}{|B_{r}(x)|}\leq{}\varepsilon_{0},
  \]
 where  \(H:=\bigl\{y\in\R^n: (y-x)\cdot e\le 0\bigr\}\) for some  \(e\in \mathbb S^{n-1}\), then 
\begin{equation*}
\fint_{B_{\lambda r}(x)}|\nabla{u}|^{2}\,dx\leq{}C\,\fint_{B_{r}(x)}|\nabla{u}|^{2}\,dx+ \frac {C r^{2}}{\lambda^n}\|\rho\|^2_{\infty}.
\end{equation*} 
\end{enumerate}
Moreover the constants \(C\) and \(\varepsilon_0 \) can be chosen to depend only on un upper bound on \(\beta\).
\end{proposition}
\begin{proof} We detail the proof of item (ii). Item  (i) can be obtained in a similar way and we sketch the argument at the end of the proof. Without loss of generality, by scaling and translating, we can assume \(x=0\) and \(r=1\).  Let \(\lambda \in (0,1/2)\) be given and let \(v\) the solution of 
\[
\begin{cases}
-\Div (a_H \nabla v)=\rho \qquad &\text{in \(B_{1/2}\)}\\
v=u& \text{on \(\partial B_{1/2}\).}
\end{cases}
\]
where \(a_H=\beta \1_H+\1_{H^c}\). In particular, \(w=(u-v)\in W^{1,2}_0(B_{1/2})\) and 
\[
-\Div (a_H \nabla w)=-\Div ((a_E-a_H)\nabla u).
\]
By testing the above equation  with \(w\) and using Young inequality we get
\[
\begin{split}
\int_{B_{1/2}} |\nabla u-\nabla v|^2\le \int_{B_{1/2}}(a_E-a_H)^2 |\nabla u|^2 \le   (\beta-1)^2 \int_{(E\Delta H)\cap B_{1/2}} |\nabla u|^2.
\end{split}
\]
Exploiting the higher integrability of Lemma \ref{l:higherintegrability} we then get 
\[
\begin{split}
\int_{B_{1/2}} |\nabla u-\nabla v|^2&\le (\beta-1)^2 |(E\Delta H)\cap B_1|^{1-\frac{1}{p}}\left (\int_{B_{1/2}} |\nabla u|^{2p}\right)^{\frac{1}{p}}
\\
&\le C(n,\beta) |(E\Delta H)\cap B_1|^{1-\frac{1}{p}}\int_{B_{1}} |\nabla u|^2\le  C(n,\beta) \eps_{0}^{1-\frac{1}{p}}\int_{B_{1}} |\nabla u|^2.
\end{split}
\]
Since the decay estimate \eqref{e:decH}  apply to \(v\), we can argue as in the proof of \ref{e:una_sola} to obtain
\[
\fint_{B_{\lambda}} |\nabla u|^2\le C\fint_{B_1} |\nabla u|^2+\frac{C  \eps_{0}^{1-\frac{1}{p}}}{\lambda^n}\int_{B_{1}} |\nabla u|^2+\frac{C  \|\rho\|_{\infty}^2}{\lambda^n}.
\]
Choosing \(\eps_0=\eps_0(n, \lambda)\ll \lambda \) sufficiently small we conclude the proof of (ii). The proof of (i) can be obtained in the same way by comparing \(u\) to a solution of \(-\Delta u=\rho\) (or \(-\beta \Delta u=\rho\)) and by using \eqref{e:decA}.
\end{proof}
\subsection{Density estimates}

In this section we establish scaling invariant  upper and lower bounds for the perimeter and for the measure of a minimizer in balls. We also establish  an universal upper bound for the normalized Dirchlet energy of the  minimizer  of  \(u_E\). We start with the following lemma which is a simple consequence of the outward minimizing property of \(E\) established in Lemma \ref{in_out} (ii). 

\begin{proposition}\label{p:densityestimate1}
Let \(A>0\), and let  \(\beta, K, Q\) be  controlled by \(A\)  and \(R\ge 1\). Then there exist universal  constants \(C_{\mathrm{o}}\) and \(r_{\mathrm{o}}\) such that, if  \(E\) is a minimizer of  \eqref{e:problem},  \(r\in (0,r_{\mathrm{o}})\), then\footnote{Here and in the sequel we will always work with the representative of \(E\) such that
 \[
 \partial E=\Biggl\{x:   \frac{|B_r(x)\setminus E|}{|B_r(x)|} \cdot \frac{|B_r(x)\cap E|}{|B_r(x)|}>0\quad \text{for all \(r>0\)}\Biggr\},
 \] 
 see \cite[Proposition 12.19]{Maggi12}.}
\begin{equation}\label{e:perimeterdens-1}
P(E,B_{r}(x))\le C_{\mathrm{o}}r^{n-1} \qquad \text{for all   \(x \in \partial E\) and \(r\in (0,r_{\mathrm{o}})\),}
\end{equation}
and
\begin{equation}\label{e:volumedens-1}
  \frac{|B_r(x)\setminus E|}{|B_r(x)|}\ge \frac{1}{C_{\mathrm{o}}}   \qquad \text{for all  \(x \in E^c\) and \(r\in (0,r_{\mathrm{o}})\),}
\end{equation}
\end{proposition}

\begin{proof}
We let \(\Lambda_2\) and \(\bar r_2\) be the constants appearing in Lemma \ref{in_out} we take \( r_{\textrm{o}}\le \bar r_2\). For \(r\le r_{\textrm{o}}\), we plug \(F=E\cup B_r(x)\) in \eqref{FoE} and we obtain, after simple manipulations, 
\[
P(E, B_r(x))\le \Hf^{n-1}(\partial B_r(x)\setminus E) +\Lambda_2 |E\setminus B_r(x)|\le n\omega_n r^{n-1}+\Lambda_2 \omega_nr^n.
\]
Hence, assuming that \(\Lambda_2 r_{\textrm{o}}\le 1\), we immediately get \(P(E, B_r(x))\lesssim r^{n-1}\). To obtain the lower density bound for \(E^c\) we set \(m(r):=|B_r(x)\setminus E|\) and we use the isoperimetric inequality to deduce 
\[
\begin{split}
m(r)^{\frac{n-1}{n}}&=|B_r(x)\setminus E|^{\frac{n-1}{n}}\lesssim P(E\setminus B_r(x))
\\
&= P(E, B_r(x))+\Hf^{n-1}(\partial B_r(x)\setminus E)
\\
& \lesssim  \Hf^{n-1}(\partial B_r(x)\setminus E)+ |E\setminus B_r(x)|
\\
&\lesssim  m'(r)+ m(r),
\end{split}
\]
where we have used that, by co-area formula  \(m'(r)= \Hf^{n-1}(\partial B_r(x)\setminus E)\). If we choose \(r_{\textrm{o}}\) such that \(Cm(r)^{\frac {1}{n}}\le C(n\omega_n)^{\frac {1}{n}} r_{\textrm{o}}\le 1/2\) where \(C\) is the implied universal constant in the above estimate, we  obtain
\[
m(r)^{\frac{n-1}{n}}\lesssim  m'(r).
\]
Since \(x\in \partial E\), \(m(r)>0\) for all \(r>0\) then the above inequality implies that 
\[
\frac{\textup{d}}{\textup{d}r} m(r)^{\frac{1}{n}}\gtrsim 1 \qquad \textrm{for all \(r\in (0, r_{\textrm{d}})\)}.
\]
Hence \(m(r)\gtrsim r^n\) and this concludes the proof.
\end{proof}

The next lemma establish an universal bound on the normalized Dirichlet integral.

\begin{lemma}\label{lbound}
Let \(A>0\), and let  \(\beta, K, Q\) be  controlled by \(A\)  and \(R\ge 1\).
Then there exists a universal  constant \(C_{\mathrm{e}}\) such that,  if  \(E\) is a minimizer of  \eqref{e:problem}, then for all \(x\in \overline{B_R}\),
\begin{equation}\label{e:diruniv}
Q^2D_{E}(x,r)=\frac{Q^2}{r^{n-1}} \int_{B_r(x)}|\nabla u|^2\,dx \le C_{\mathrm{e}}.
\end{equation}
\end{lemma}%
\begin{proof} The estimates is clearly true if \(r\ge r_0\) where \( r_0=r_0(n,A)\)  (recall that \(Q^2\int |\nabla u_E|^2\le \F_{\beta, K,Q}(E)\lesssim 1\)). Hence we can assume that \(r\le r_0\ll 1\). We claim the following: there exist constants  \(\lambda=\lambda(n,A)\in (0,1/2)\),  \(C=C(n,A)\)  and \(r_0=r_0(n,A)\)  such that 
\begin{enumerate} [(a)]
\item  
If \(x\in \partial B_R\)  and \(r\le r_0\),  then
\begin{equation}\label{e:casobordo}
 Q^2D_{E}(x,\lambda r)\le \frac{1}{2} Q^2D_{E}(x, r)+C.
\end{equation}
\item If \(x\in B_R\)  and \(r\le \min\bigl\{\dist(x, \partial B_R), r_0/2\bigr\}\), then 
\begin{equation}\label{e:casodentro}
 Q^2D_{E}(x,\lambda r)\le \frac{1}{2} Q^2D_{E}(x, r)+C.
\end{equation}
\end{enumerate}
Let \(\eps\ll 1\) to be fixed and  let \(r_0=r_0(\eps)\ll \bar r_1\) where \(\bar r_1\) is the constant in  Proposition  \ref{t2} and such that the following holds true
\begin{equation}\label{e:almosthalf}
x\in\ \partial B_R\quad\text{and \(r\le r_0\)}\quad\Longrightarrow\quad \frac{|(B_R\cap B_r(x))\Delta H_{x}|}{|B_r(x)|}\le \eps,
\end{equation}
where \(H_{x}:=\{y: (y-x)\cdot x\le 0\}\) is the supporting half space of \(B_R\) at \(x\). Note that since the curvatures of \(\partial B_R\) are universally bounded (recall that \(R\ge 1\)), this can be achieved by choosing \(r_0\) small only in dependence of \(\eps\).

Let now \(x\in \overline{B_R}\) and \(r\le r_0\) be a radius satisfying either condition (a) (if \(x\in \partial B_R\)) or condition (b) (if \(x\in B_R\))  above.  Let \((u_E, \rho_E)\) be the minimizers for \(\G(E)\) and consider 
\[
F=(E\cup B_r(x))\cap B_R.
\]
We define  \(u\) to be the solution of
\begin{equation}\label{e:u}
-\Div(a_F \nabla u)=\rho_E.
\end{equation}
Note that \((u,\rho_E)\in \A(F)\) since \(F\supset E\). Hence, by  Proposition  \ref{t2},
\[
\begin{split}
P(E)+Q^2\Bigl( \int_{\R^n}a_E & |\nabla u_E|^2+K \int_{\R^n}\rho_E^2\Bigr)\\
&\le P(F)+Q^2\Bigl( \int_{\R^n}a_F |\nabla u|^2+K \int_{\R^n} \rho_E^2\Bigr)+\Lambda_1 |F\setminus E|
\\
&\le P(E\cup B_{r}(x))+Q^2\Bigl( \int_{\R^n}a_F |\nabla u|^2+K \int_{\R^n} \rho_E^2\Bigr)+\Lambda_1 |B_r(x)|,
\end{split}
\]
where we have used that \(F\setminus E \subset B_r(x)\) and that  \(P(F)\le  P(E\cup B_{r}(x))\), by the convexity of \(B_{R}\). Rearranging terms we get 
\[
Q^2\Bigl( \int_{\R^n}a_E |\nabla u_E|^2-\int_{\R^n}a_F |\nabla u|^2\Bigr)\le P(E\cup B_{r}(x))-P(E)+\Lambda_1 |B_r(x)|\lesssim r^{n-1}.
\]
Recall now that \(u_E\) solves 
\[
-\Div(a_E \nabla u_E)=\rho_{E},
\]
and we use \eqref{e:resto} in  Lemma \ref{D_Lemma_1} to infer that
\begin{equation}\label{e:stanco}
\int (a_F-a_E)|\nabla u|^2 \le \int_{\R^n}a_E |\nabla u_E|^2-\int_{\R^n}a_F |\nabla u|^2\lesssim \frac{r^{n-1}}{Q^2}.
\end{equation}
Since
\[
-\Div (a_E \nabla (u_E-u))=-\Div ((a_F-a_E)\nabla u),
\]
by testing with \(u_E-u\) and by  Young inequality  we get 
\begin{equation*}
Q^2\int_{\R^n} |\nabla u_E-\nabla u|^2\le Q^2 \int (a_F-a_E)^2|\nabla u|^2\lesssim r^{n-1},
\end{equation*}
where the last inequality follows  from~\eqref{e:stanco}.

We want now apply Lemma \ref{1decay} to \(u\). Note that since
\[
F\cap B_r(x)=B_r(x)\cap B_R,
\]
then the assumption are satisfied both in case (a) (thanks to \eqref{e:almosthalf}) and in case (b) (since \(B_r(x)\subset B_R\)).  Hence, given \(\lambda \in (0,1/2)\), we have:
\begin{equation}\label{e:ta}
\begin{split}
\frac{1}{(\lambda r)^{n-1}}&\int_{B_{\lambda r}(x)} |\nabla u_E|^2\le \frac{2}{(\lambda r)^{n-1}}\int_{B_{\lambda r(x)}} |\nabla u-\nabla u_E|^2+\frac{2}{(\lambda r)^{n-1}}\int_{B_{\lambda r(x)}} |\nabla u|^2
\\
&\le \frac{2}{(\lambda r)^{n-1}}\int_{B_{\lambda r}(x)} |\nabla u-\nabla u_E|^2+\frac{C \lambda}{r^{n-1}}\int_{B_r(x)} |\nabla u|^2+\frac{C r^2\|\rho_E\|_{\infty}}{\lambda^{n-1}}
\\
&\le \frac{C}{\lambda^{n-1}}\frac{1}{r^{n-1}}\int_{B_r(x)} |\nabla u-\nabla u_E|^2+\frac{C \lambda}{r^{n-1}}\int_{B_r(x)} |\nabla u_E|^2+\frac{C r^2\|\rho_E\|_{\infty}}{\lambda^{n-1}}\,,
\end{split}
\end{equation}
for a constant \(C=C(n,A)\) provided \(\eps\) (and thus \(r_0\)) is chosen sufficiently small. Since by  \eqref{e:rhobound} \(\|\rho_{E}\|_{\infty}\lesssim 1\),  we deduce from \eqref{e:ta} that
\begin{equation} \label{e:D_E}
Q^2 D_E(x,\lambda r)\le C\lambda \,Q^2 D_E(x, r)+\frac{C(n,A)}{\lambda^{n-1}}.
\end{equation}
Now  choosing \(\lambda=\lambda(n,A)\) such that \(C\lambda=1/2\) we conclude the proof of the claim. Note that this   fixes \(\eps\) and thus \(r_0\) as functions depending only on \(n\) and \(A\).

To conclude the proof we have to show that (a) and (b) above implies that 
\[
S:=\sup_{y\in \overline{B_R}} \sup_{0<s\le r_0} Q^2 D_E(y, s)\le C(n,A).
\]
We first assume that \(S<+\infty\) and show that we can bound it by a universal constant. Let \(\bar y\in  \overline{B_R}\) and \(\bar s\in 0<s\le r_0\)  be such that
\[
\frac{3S}{4}\le Q^2 D_E(\bar y, \bar s)
\]
Let us distinguish a few cases:

\medskip
\noindent
\(\bullet\) \emph{Case 1: \(\bar y\in \partial B_R\).} If \(\bar s \le \lambda r_0\), \eqref{e:casobordo} implies that
\[
\frac{3S}{4}\le Q^2 D_E(\bar y , \bar s)\le\frac{1}{2} Q^2 D_E\Bigl(\bar y ,\frac{ \bar s}{\lambda}\Bigr)+C\le  \frac{1}{2} S+C,
\]
and we are done. On the other end if \(\bar s \ge \lambda r_0\), then 
\[
\begin{split}
\frac{3S}{4}\le Q^2 D_E(\bar y , \bar s)&\le \frac{ Q^2}{(\lambda r_0)^{n-1}} \int_{\R^n} |\nabla u_E|^2
\\
&\le \frac{ 1}{(\lambda r_0)^{n-1}}\F_{\beta,K,Q}(E)\le C(n,A).
\end{split}
\]

\medskip
\noindent
\(\bullet\) \emph{Case 2: \(\bar y\in B_R\).} If \(\bar s \le \lambda \min\{ \dist(\bar y, \partial B_R), r_0/2\} \), we can use  \eqref{e:casodentro} and we argue as in the first part of Case 1. If \(\bar s \ge  \lambda  r_0/2 \)  we argue instead as in the second part of Case 1 to conclude. We are thus left to consider the case
\[
\lambda \dist(\bar x, \partial B_R)\le \bar s \le \lambda  r_0/2.
\]
In this case \(B_{\bar s} (\bar y)\subset B_{r_0} (\bar y)\), \(\bar y\in \partial B_{R}\) and 
\[
\frac{3S}{4}\le Q^2 D_E(\bar y , \bar s)\le \frac{1}{2} S+C.
\]
Thus we are done. 

To show that one can actually assume that \(S<+\infty\) one can consider 
\[
S_\delta=\sup_{y\in \overline{B_R}} \sup_{\delta \le s\le r_0} Q^2 D_E(y, s)\le C(n,A) \delta^{1-n}
\]
and argue as above to show that \(S_\delta\le C(n,A)\). Letting \(\delta\to 0\) we conclude the proof.
\end{proof}

We are now ready to complete the proof of density and perimeter estimates.

\begin{proposition}\label{p:densityestimate2}
Let \(A>0\), and let  \(\beta, K, Q\) be  controlled by \(A\)  and \(R\ge 1\).
Then there exist universal constants \(C_{\mathrm{i}}\) and \(\bar r_{\mathrm{i}}\) such that, if  \(E\) is a minimizer of  \eqref{e:problem},  then   
\begin{equation}\label{e:perimeterdens}
P(E,B_{r}(x))\ge \frac{r^{n-1}}{C_{\mathrm{i}}} \qquad \text{for all  \(x \in \partial E\) and \(r\in (0,\bar r_{\mathrm{i}})\),}
\end{equation}
and
\begin{equation}\label{e:volumedens}
  \frac{|B_r(x)\cap E|}{|B_r(x)|}\ge \frac{1}{C_{\mathrm{i}}}  \qquad \text{for all  \(x \in E\) and \(r\in (0,\bar r_{\mathrm{i}})\),}
\end{equation}
\end{proposition}

\begin{proof}  We start showing the validity of \eqref{e:perimeterdens} and we divide the proof in few steps.

\medskip 
\noindent
\(\bullet\) \textit{Step 1}: We claim that for every \(\lambda\in (0,1/4)\), there exist \(\varepsilon_1=\varepsilon_1(\lambda, ,A)\), \(C_1=C_1(n,A)\) and and \(\bar r=\bar r(n,A, \lambda)\) such that if
\[
P(E, B_r(x))\le \varepsilon r^{n-1}\qquad \varepsilon \le \varepsilon_1\,\quad r\le \bar r\,,
\]
then, 
\begin{equation}\label{e:idecay}
P(E, B_{\lambda r}(x))+Q^2\int_{B_{\lambda r}(x)} |\nabla u_E|^2\le C_1\lambda^n \Biggl(P(E, B_{r}(x))+Q^2\int_{B_{ r}(x)} |\nabla u_E|^2+r^n\Biggr).
\end{equation}
For the ease of notation let us assume that \(x=0\). Let \(\lambda \in (0,1/4)\) be fixed. By the relative isoperimetric inequality
\[
\Biggl(\min\Biggl\{\frac{|E\cap B_r|}{|B_r|}, \frac{| B_r\setminus E|}{|B_r|}\Bigg\}\Biggr)^{\frac{n-1}{n}}\le C(n) \frac{P(E, B_{ r}(x))}{r^{n-1}}\lesssim \eps.
\]
By \eqref{e:volumedens-1} and  by choosing \(\eps_1, \bar r\ll1\)  we get  
\begin{equation}\label{e:permis}
\frac{|E\cap B_r|}{|B_r|}\le C(n)  \Biggl(\frac{P(E, B_{r}(x))}{r^{n-1}}\Biggr)^{\frac{1}{n-1}} \frac{P(E, B_{ r}(x))}{r^{n-1}}\lesssim \eps^{\frac{1}{n-1}}\frac{P(E, B_{ r}(x))}{r^{n-1}}.
\end{equation}
Let us choose  $t\in{}(\lambda{}r,2\lambda{}r)$ such that 
\begin{equation}\label{e:derper}
\begin{split}
\Hf^{n-1}(E\cap \partial B_t) &\leq{}\fint_{\lambda{}r}^{2\lambda{}r}\Hf^{n-1}(E\cap \partial B_s)\,ds
\\
&\le \frac{|E\cap B_{2\lambda r}|}{\lambda r}\le C(n,\lambda) \eps^{\frac{1}{n-1}}P(E, B_{ r}(x)).\\
\end{split}
\end{equation}
By testing \eqref{FoE} with $F=E\setminus{}B_{t}(x)$ we obtain
\begin{equation} \label{Di0}
P(E,B_{t})\leq{}\mathcal{H}^{n-1}(E\cap{}\partial{}B_{t})+\Lambda_2 |E\cap{}B_{t}(x)|+\Lambda_2 Q^2\int_{E\cap{}B_{t}(x)}|\nabla{}u_E|^{2}\,dx,
\end{equation}
which together with \eqref{e:derper} and  recalling that \(t\in (\lambda r, 2\lambda r)\), implies that
\begin{multline}\label{e:plambda}
P(E,B_{\lambda r})+Q^2\int_{B_{\lambda r}(x)} |\nabla u_E|^2
\\
\le C(n,\lambda) \eps^{\frac{1}{n-1}}P(E, B_{ r}(x))+(\Lambda_2+1) Q^2 \int_{B_{2\lambda r}(x)} |\nabla u_E|^2+\Lambda_2 |B_{2\lambda r}|.
\end{multline}
If we now choose \(\eps_1=\eps_1(\lambda)\ll1\), \eqref{e:permis} allow to apply Proposition \ref{1decay} (i). Hence by also choosing \(\bar r\ll \lambda\) we deduce that
\begin{equation}\label{e:dlambda}
\begin{split}
 \int_{B_{2\lambda r}(x)} |\nabla u_E|^2&\le C(n,A)\lambda^n \Biggl(\int_{B_{ r}(x)} |\nabla u_E|^2+\frac{\bar r^{2}}{\lambda^n}r^n\Biggr)
 \\
 &
 \le C(n,A)\lambda^n \Biggl(\int_{B_{ r}(x)} |\nabla u_E|^2+r^n \Biggr)\,,
 \end{split}
\end{equation}
where we have used that by \eqref{e:rhobound}, \(\|\rho_E\|_{\infty} \lesssim 1\). By gathering equations \eqref{e:plambda} and \eqref{e:dlambda} we then get
\begin{multline*}
P(E,B_{\lambda r})+Q^2\int_{B_{\lambda r}(x)} |\nabla u_E|^2
\\
\le C(n,\lambda)\varepsilon^{\frac{1}{n-1}}P(E, B_r(x))+C(n,A)\lambda^n\Biggl(Q^2\int_{B_{ r}(x)} |\nabla u_E|^2+r^n\Biggr).
\end{multline*}
If we choose \(\eps_1=\eps_1(n,A, \lambda)\ll 1\) such that \(C(n,\lambda)\varepsilon^{\frac{1}{n-1}}\le \lambda^n\) the above inequality implies \eqref{e:idecay}.

\medskip

\noindent
\(\bullet\) \textit{Step 2}: We now prove the validity of \eqref{e:perimeterdens}. By density it is enough to prove it at all \(x\in \partial^*E\). Again we set coordinates so that \(x=0\). Let us choose \(\lambda=\lambda(n,A) \in (0,1/4)\) such that \(C_1\lambda \le 1/2\) where \(C_1\) is the constant appearing in \eqref{e:idecay} and let \(\bar r\) and \(\eps_1\) be the corresponding constants (which now depend only on \(A\) and \(n\)). We claim that
\begin{equation}\label{e:inizio}
P(E, B_r(x))+Q^2\int_{B_{r}}|\nabla u_E|^2 \ge \frac{\eps_1}{2}\, r^{n-1} \qquad \text{for all \(r\le \min\{r_1, \eps_1/2\}\)}.
\end{equation}
Indeed otherwise, by \eqref{e:idecay} and the choice of \(\lambda\)
\begin{multline*}
P(E,B_{\lambda r}(x))+Q^2\int_{B_{\lambda r}}|\nabla u_E|^2 
\\
\le \frac{\lambda^{n-1}}{2} \Biggl(P(E,B_{r}(x))+Q^2\int_{B_{r}}|\nabla u_E|^2 +\frac{\eps_1}{2} \,r^{n-1}\Biggr)\le \frac{\eps_1}{2} (\lambda r)^{n-1}.
\end{multline*}
We can thus  iterate the above estimate and deduce  that
\[
\liminf_{r \to 0} \frac{P(E, B_r)}{r^{n-1}}=0\,,
\]
in contradiction with the assumption that \(0\in \partial^*E\). Let now  \(\bar \lambda\ll \eps_1\)   to be chosen   where \(\eps_1\)  is   the constant obtained above.  Let \(\eps_2\) and \(r_2\) be the constants corresponding to \(\bar \lambda\) in Step 1.  We claim that if we choose \(\bar \lambda\) small enough depending only on \(n\) and \(A\) then  
\begin{equation}\label{e:peps}
P(E,B_{r})\ge \eps_2 r^{n-1}\qquad \text{for all \(r\le r_3\),}
\end{equation}
where \(r_3\ll \min\{r_2, r_1\}\) will depend only on \(n\) and \(A\). Indeed otherwise we can apply Step 1, \eqref{e:perimeterdens-1}, and  Lemma \ref{lbound} to get
\[
\begin{split}
P(E,B_{\bar \lambda r}(x))+Q^2\int_{B_{\bar \lambda r}}|\nabla u_E|^2 & \le C(n,A)\bar \lambda^{n} \Biggl(P(E,B_{ r}(x))+Q^2\int_{B_{r}}|\nabla u_E|^2 +r^n\Biggr)
\\
&\le \bar C(n,A)\bar{\lambda} (\bar\lambda r)^{n-1},
\end{split}
\]
where \(\eps_2\ll \eps_1\) and \(r_2\ll r_1\) are universal constants. If \(\bar \lambda\) is chosen so that \( \bar C(n,A)\bar\lambda
\le \eps_1/4\) this contradicts \eqref{e:inizio} and thus proves \eqref{e:perimeterdens} with \(c_{\mathrm{i}}\le \eps_2\).

\medskip

\noindent
\(\bullet\) \textit{Step 3}: We now prove the validity of \eqref{e:volumedens}.  Assume indeed that 
\[
\frac{|E\cap B_r|}{|B_r|}\le  \varepsilon_4 \qquad\text{for \(r\le r_4\)},
\]
with \( \varepsilon_4, r_4\ll1\) to be fixed only in term of \(n\) and \(A\). Then, by \eqref{l:higherintegrability} and \eqref{e:diruniv}, for all \(s\in (r/4,r/2)\)
\begin{equation}\label{e:dirpiccolino}
\begin{split}
Q^2\int_{B_{s}}|\nabla u_E|^2&\le Q^2|E\cap B_s|^{1-\frac{1}{p}}\Biggr(\int_{B_{s}}|\nabla u_E|^2p\Biggr)^{\frac{1}{p}}
\\
&\lesssim Q^2\Biggl(\frac{|E\cap B_r|}{|B_r|}\Biggr)^{1-\frac{1}{p}}\int_{B_{2s}}|\nabla u_E|^{2p}\lesssim \eps_4^{1-\frac{1}{p}}r^{n-1}\lesssim \eps_4^{1-\frac{1}{p}}s^{n-1}.
\end{split}
\end{equation}
Moreover, by co-area formula, there exists   \(s\in (r/4,r/2)\) such that
\begin{equation}\label{e:perpiccolo}
\Hf^{n-1}(E\cap B_s)\le \fint_{r/4}^{r/2}\Hf^{n-1}(E\cap B_t)dt\le \frac{4|E\cap B_r|}{r}\lesssim \varepsilon_4 r^{n-1}\lesssim \varepsilon_4 s^{n-1}.
\end{equation}
By testing \eqref{FinE} with \(E\setminus B_s\) we get
\[
P(E,B_s)\le  \Hf^{n-1}(E\cap B_s)+\Lambda_2|B_s|+Q^2\Lambda_2\int_{B_s}|\nabla u_E|^2,
\]
which together with \eqref{e:dirpiccolino} and \eqref{e:perpiccolo} and provided \(r_4\ll \eps_4\ll 1\) implies
\[
P(E,B_s)\le C  \eps_4^{1-\frac{1}{p}}s^{n-1},
\]
for a suitable universal constant \(C\). Choosing \(\eps_4\) small with respect to \(\eps_2\) we get
\[
P(E,B_s)\le \varepsilon_2 s^{n-1},
\]
in contradiction with \eqref{e:peps}.

\end{proof}

\section{Decay of the excess}\label{s:decay}

In this section we prove Theorem~\ref{thm:maineps}.  Since the seminal works of De Giorgi and Almgren, \cite{De-Giorgi60, Almgren68}  the proof  is based on an excess decay theorem, namely

\begin{theorem}[Excess improvement] \label{excessimprove}
Let \(A>0\), and let  \(\beta, K, Q\) be  controlled by \(A\)  and \(R\ge 1\). There exists a universal constant \(C_{\mathrm {dec}}>0$ such that for all   \(\lambda\in (0,1/4)\) there exists $\varepsilon_{\mathrm{dec}}=\varepsilon_{\mathrm{dec}}(n,A,\lambda)>0$ satisfying the following:   if \(E\) is a minimizer of \eqref{e:problem} and 
\begin{equation*}
x\in\partial E,\quad  r+Q^2D_{E}(x,r)+\e_E(x,r)\leq{}\varepsilon_{\mathrm{dec}},
\end{equation*}
then
\begin{equation}\label{e:superdecay}
Q^2D_{E}(x,\lambda r)+\e_E(x,\lambda r)\leq{} C_{\mathrm{dec}}\lambda \Bigl(\e_E(x,r) +Q^2D_{E}(x,r)+r\Bigr).
\end{equation}
\end{theorem}
Where we recall the definition of spherical excess 
 \[
\e_E(x,r)=\inf_{\nu \in \mathbb S^{n-1}} \frac{1}{r^{n-1}} \int _{\partial^{*}E\cap B_{r}(x)}\frac{|\nu_{E}(y)-\nu|^{2}}{2}d \Hf^{n-1}(y),
\]
and of the normalized Dirichlet energy
\[
D_E(x, r)=\frac{1}{r^{n-1}} \int_{B_{r}(x)} |\nabla u_{E}|^{2}. 
\]
As it is customary, the proof of the above theorem is based on ``harmonic approximation'' technique. More  precisely we will go through the following steps: 
\begin{enumerate}[(i)]
\item In the  small excess regime,  the boundary of \(E\) can be well approximated by the graph of a Lipschitz function \(f\) with Dirichlet energy bounded by the excess.
\item  If the excess and the normalized Dirichlet of \(u_E\) are small, \(f\) is almost harmonic.
\item Almost harmonicity of \(f\) implies closedness to an harmonic function \(g\) in the \(L^2\) topology. By classical estimates for harmonic functions, an \(L^2\) type of excess of \(g\) decays  which in turn implies the decay of the \emph{flatness} \(\f\)  of  \(E\), see \eqref{flatness} below for the definition.
\item Via a Caccioppoli type inequality, the decay of the flatness   can be transferred to the decay of the excess.
\item Via  Proposition \ref{1decay} (ii), the decay of the excess implies the decay of the normalized Dirchlet energy.
\end{enumerate}

Usually,  Step (i) is obtained by reproducing at most points and at  all scale an height type bound for \(\partial E\) in the small excess regime and it thus relies on the scaling invariance of the problem studied. Step (ii) and (iv) are obtained by simple comparison arguments and Step (iii) is based on a compactness argument together with the classical regularity theory for harmonic functions.

In our situation the problem does not enjoy of a nice scaling behaviour, due to the global constraint \(\int\rho_E=1\). However, the local estimates obtained in the previous section are exactly what we need to carry on the proof of Step (i), see Lemma \ref{fascia} below. Since beside this  fact, most of the proofs of the needed lemmas are almost verbatim adaptation of those present in literature, we will not detail all of them and we will just focus on the key points and on the main differences.

\subsection{Lipschitz approximation}
In this subsection we prove the Lipschitz approximation lemma. Let us first fix a few notations that will be useful through all the section.

For  $\nu\in{}\mathbb{S}^{n-1}$ we  let $\mathbf{p}^{\nu}(x):=x-(x\cdot{}\nu)\,\nu$ and $\mathbf{q}^{\nu}(x):=(x\cdot{}\nu)\,\nu$ be, respectively, the orthogonal projection onto the plane $\nu^\perp$ and the projection on $\nu$. For simplicity we denote $\mathbf{p}(x):=\mathbf{p}^{e_{n}}(x)$ and $\mathbf{q}(x):=\mathbf{q}^{e_{n}}(x)=x_{n}$.

We define the cylinder with center at $x\in{}\R^{n}$ and radius $r>0$ with respect to the direction $\nu\in{}\mathbb{S}^{n-1}$ as \[
\mathbf{C}(x,r,\nu):=\bigl\{x\in\R^{n}\,:\,|\mathbf{p}^{\nu}(x-x_{0})|<r\,\mbox{,}\,|\mathbf{q}^{\nu}(x-x_{0})|<r\bigr\}.
\]
 We will write $\mathbf{C}_{r}:=\mathbf{C}(0,r,e_{n})$ and $\mathbf{C}:=\mathbf{C}_{1}$. We will also will denote the (\(n-1\))-dimensional disk centered at \(y\) and of radius \(r\) by 
 \[
 \mathbf{D}(y,r):=\bigl\{y\in{}\R^{n-1}: |y-y_{0}|<r\bigr\}.
 \]
 For simplicity we will write $\mathbf{D}_{r}:=\mathbf{D}(0,r)$ and $\mathbf{D}:=\mathbf{D}(0,1)$. We also recall the definition of \emph{cylindrical excess} in a  direction \(\nu\in \mathbb S^{n-1}\) to be 
 \[
 \e_{E}(x,r,\nu)=\frac{1}{r^{n-1}}\int_{\mathbf{C}(x,r,\nu{})\cap{}\partial^{*}E}\frac{|\nu_{E}(y)-\nu|^{2}}{2}\,\,d\mathcal{H}^{n-1}(y)\,,
 \] 
 so that 
 \begin{equation}\label{e:exsexc}
 \e_{E}(x,r)\le \inf_{\nu \in S^{n-1}} \e_{E}(x,r,\nu).
 \end{equation}

The following height bound is crucial in the sequel. Note that it does not require any minimality property on \(E\), only the validity of inequality \eqref{alt1} at all scales.

\begin{lemma} \label{fascia} 
Let \(C>0\),  there exists  an increasing function $\omega_C:(0,1)\rightarrow{}\R$ with $\omega_C(0^{+})=0$ depending only on \(C\) , such that  every $E\subseteq{}\R^{n}$ of finite perimeter in $\mathbf{C}(x,2r)$  such that 
\begin{itemize}
\item[\textup(i)]\( x\in \partial E\),
\item[\textup(ii)] for all \(y\in \partial E\) and \(s\) such that \(B_{s}(y)\subset {}\mathbf{C}(x,2r)\)
\begin{equation} \label{alt1}
\frac{s^{n-1}}{C}\leq{}P(E,B_{s}(x))\leq C s^{n-1},
\end{equation}
\end{itemize}
satisfies the following
\begin{align}
\e_E(x,2r,e_{n})<t \qquad \Longrightarrow\;
&\sup_{y\in{}\C(x,r)\cap\partial E}|\mathbf{q}(y-x)|\leq{}\omega_{C}(t)r , \label{omega1}\\
&\big|\{y\in{}\mathbf{C}(x,r)\cap E\,:\,\mathbf{q}(y-x)>\omega_{C}(t)r\}\big|=0, \label{omega2}\\
&\big|\{y\in{}\mathbf{C}(x,r)\setminus E\,:\mathbf{q}(y-x)<-\omega_{C}(t)r\}\big|=0. \label{omega3}
\end{align}
\end{lemma}
\begin{proof} Note that the assumptions are scaling and translation  invariant, hence we can assume that \(x=0\) and \(r=1\). For every $t\in{}(0,1)$ let
\begin{equation*}
\mathcal{M}_{t}:=\bigl\{\text{sets of finite perimeter satisfying \(\e(E,0,2,e_{n})<t\), (i) and (ii)} \bigr\}.
\end{equation*}
For every $E\subseteq{}\R^{n}$ let us call
\begin{equation}
\begin{split}
&h_{E}:=\sup_{x\in{}\mathbf{C}\cap{}\partial{}E}|\mathbf{q}x|,\\
&g_{E}:=\inf\big\{s\in[0,1]\,:\,|\{x\in{}\mathbf{C}\cap{}E\,:\,\mathbf{q}x>s\}|=0\big\}\;\;\text{ and }\\
&f_{E}:=\inf\big\{s\in[0,1]\,:\,|\{x\in{}\mathbf{C}\setminus{}E\,:\,\mathbf{q}x<-s\}|=0\big\}.
\end{split}
\end{equation}
Define the functions $\omega_{1},\omega_{2},\omega_{3}:(0,1)\rightarrow{}\R$ as 
\begin{equation}
\omega_{1}(t):=\sup_{E\in{}\mathcal{M}_{t}}h_{E},\;\;\omega_{2}(t):=\sup_{E\in{}\mathcal{M}_{t}}g_{E}\;\;\text{ and }\;\;
\omega_{3}(t):=\sup_{E\in{}\mathcal{M}_{t}}f_{E}\mbox{.}
\end{equation} 
Let $\omega_{C}:=\max\{\omega_{1},\omega_{2},\omega_{3}\}$. Notice that $\omega_{ C}$ is increasing since it is the maximum of  increasing functions and by definition it satisfies~\eqref{omega1},~\eqref{omega2}, and~\eqref{omega3}.
Let us prove that $\omega_{C}(0^{+})=0$. Assume by contradiction that  $\lim_{t\to{}0^{+}}\omega_{C}(t)>0$ then there exist a there exists a sequence $t_{k}\searrow{}0$ and $L>0$ such that $\omega_{C}(t_{k})>L$ for all \(k\). We now distinguish three cases.

\medskip
\noindent
\textit{Case 1}: Up to subsequences $\omega(t_{k})=\omega_{1}(t_{k})$ for every $k\in\mathbb{N}$. For every $k$ there exists $E_{k}\in\mathcal{M}_{t_{k}}$ such that $h_{E_{k}}\geq{}L$. By ~\eqref{alt1}) up to subsequences there exists a set of finite perimeter $E\subseteq{}\R^{n}$ such that  $E_{k}\cap{}\mathbf{C}_{r}\rightarrow{}E\cap \mathbf{C}_{r}$ whenever \(r<2\) and 
\begin{equation}
\lim_{k\to{}+\infty{}}\mathbf{e}(E_{k},0,2,e_{n})=0\mbox{.}
\end{equation}
Now take $\mathbf{C}_{s}\subset{}\mathbf{C}_{r}\subset{}\mathbf{C}_{2}$ with $s>1$. By the lower semicontinuity of the excess we obtain that $\mathbf{e}(E,0,s,e_{n})=0$. Moreover let  $\{x_{k}\}_{k\in{}\mathbb{N}}$ be a sequence such that $x_{k}\in{}\partial{}E_{k}\cap{}\mathbf{C}$ and let us assume that \(x_k\to x\). By (ii) one easily deduce that 
\[
P(E,B_r(x))\ge \frac{r^{n-1}}{C}
\]
which implies that \(x\in \partial E\) (recall that we are working with the representative of \(E\) such that \(\partial E=\spt D\1_{E}\)). This in particular implies that \(0\in \partial E\). Since $\mathbf{e}(E,0,s,e_{n})=0$ we get that \(H=\{x: \q x<0\}\). However, if \(x_k\in \partial E_k\) is such that  $|\mathbf{q}x_{k}|\geq{}L$, up to a subsequence, we can assume that \(x_k\to \bar x \in \partial E=\{x: \q x=0\}\), a contradiction.

\medskip
\noindent
\textit{Case 2}:  Up to subsequence $\omega(t_{k})=\omega_{2}(t_{k})$ for every $k\in\mathbb{N}$. Hence  for every $k$ there exists $E_{k}\in\mathcal{M}_{t_{k}}$ such that $g_{E_{k}}\geq{}L$. Note that if $\ell\in{}(0,L)$ then
\begin{equation} \label{alt2}
\big|\{x\in{}\mathbf{C}\cap{}E_{k}\,:\,\mathbf{q}x>\ell\}\big|>0\qquad\text{for all \(k\in \mathbb N\).}
\end{equation}
Hence ~\eqref{alt2} implies that, up to extracting a subsequence, 
\begin{gather}
\intertext{either}
\text{there exists \(\ell \in(0,L)\)  such that for \( k\) there exists \(x_{k}\in{}\mathbf{C}\cap\partial E_{k}\cap\{\q x>\ell\}\),}\label{alt3}
\\
\intertext{or}
\text{ \(\1_{E_{k}\cap\{\mathbf{q}x>0\}}\longrightarrow \1_{\{\mathbf{q}x>0\}}\) in \(L^{1}(\C)\).}\label{alt4}
\end{gather}
Indeed if by contradiction ~\eqref{alt3} does not hold then for every $j\gg 1 $  there exists $k_{j}\in\mathbb{N}$ such that $\mathbf{q}x\leq\frac{1}{j}$ for every $x\in{}\mathbf{C}\cap{}\partial E_{k_{j}}$. By~\eqref{alt2}, since $\{\mathbf{q}x>\frac{1}{j}\}$ is connected, then necessarily $\mathbf{C}\cap{}E_{k_{j}}\supseteq{}\{\mathbf{q}x>\frac{1}{j}\}.$ By letting $j\to{}+\infty$ we get~\eqref{alt4}.

By arguing as in Case 1, we have that \(E_k\to \{\q x\le 0\}\), hence \eqref{alt4} cannot hold. Hence~\eqref{alt3} holds, which is again in contradiction with  Case  1. 

\medskip
\noindent 
\textsl{Case 3}: Up to subsequence $\omega(t_{k})=\omega_{3}(t_{k})$ for every $k\in\mathbb{N}$. This case can be ruled out by arguing as in Case 2 (or by working with \(E^c\) which satisfies the same assumption of \(E\)). Therefore $\omega$ is the required function.
\end{proof}


Once the ``qualitative'' height bound has been established, one can repeat verbatim the proof of the Lipschitz approximation in \cite[Theorem 2.37]{Maggi12} to deduce that in the small excess regime \(\partial E\) is mostly covered by the graph of a Lipschitz function. Note that in the cited reference one has an explicit formula for \(\omega\) (namely \(\omega(t)\lesssim t^{1/(n-1)}\)) however this plays at all no role in the proof, see also \cite[Lemma 4.3]{De-PhilippisMaggi15}.

\begin{lemma}[Lipschitz approximation I] \label{lipappr}
Fix $C>0$. Then there exists \(\eps_{\mathrm{L}}=\eps_{\mathrm{L}}(n,C)>0\) and \(C_{\mathrm{L}}=C_{\mathrm{L}}(n,C)>0\) with the following property: let \(E\) be a set of finite perimeter in \(\C(x,4r)\) satisfying \(x\in \partial E\),
\[
\frac{s^{n-1}}{C}\leq{}P(E,B_{s}(y))\leq C s^{n-1}\quad\text{for all \(y\in \partial E\cap \C(x,2r)\) such that \(B_s(y)\subset \C(x,2r)\)},
\]
and
\[
\e_E(x,2r,e_n)\le \eps_{\mathrm{L}}.
\]
Then there exists a  function $f:\R^{n-1}\to{}\R$ with 
\begin{equation}\label{lip000}
\Lip (f)\leq 1,\quad \frac{1}{r^{n-1}}\int_{D_r} |\nabla f|^2\le C_{\mathrm{L}}e_E(x,2r,e_{n}), \quad  \frac{\|f\|_{\infty}}{r}\le \omega_{C}\bigl(\e_E(x,2r,e_n)\bigr),
\end{equation}
 such that, defining $\Gamma_{f}:=x+\{(z,f(z))\,:\,z\in{}\mathbf{D}_{r}\}$,
 \begin{equation}\label{lip0}
 \frac{\mathcal{H}^{n-1}\bigl((\partial E\cap \C(x,r, e_n))\Delta\Gamma_{f}\bigr)}{r^{n-1}}\leq{}C_{\mathrm{L}}\e_E(x,2r,e_{n}),
 \end{equation}
here \(\omega_C\) is the function in Lemma \ref{fascia}. 
\end{lemma}

Note that if  $E$ is a minimizer of \eqref{e:problem}, the assumption of the Lipschitz approximation lemma are satisfied with some universal constant \(C\) by  \eqref{e:perimeterdens-1} and \eqref{e:perimeterdens}. Hence we can cover most of its boundary by the graph of a Lipschitz function \(f\). Moreover a simple comparison argument implies that the laplacian of \(f\) is small in a suitable negative norm. More precisely we have the following:

\begin{proposition} [Lipschitz approximation II] \label{lipappr2}\
Let \(A>0\), and let  \(\beta, K, Q\) be  controlled by \(A\)  and \(R\ge 1\). Then there exists universal constants \(\eps_{\mathrm{lip}}\), and \(C_{\mathrm {lip}}\) and a ``universal'' increasing function (i.e. depending only on \(n\) and \(A\)) \(\omega_{\mathrm {lip}}\) with \(\omega_{\mathrm {lip}}(0+)=0\) such that if \(E\) is a minimizer of \eqref{e:problem}, \(x\in \partial E\) and 
\[
r+\e_E(x,2r,e_n)\le \eps_{\mathrm{lip}},
\]
then there exists a function \(f\) satisfying \eqref{lip000} and  \eqref{lip0} with  $C_L$ and \(\omega_C\) replaced by \(C_{\mathrm{lip}}\) and \(\omega_{\mathrm {lip}}\) respectively. Moreover 
\begin{equation} \label{quasih}
\frac{1}{r^{n-1}}\left|\int_{\mathbf{D}_{r}}\nabla f\cdot\nabla\varphi\;dz\right| \leq C_{\mathrm {lip}}\,\|\nabla \varphi\|_{\infty}\,\Biggr(\e_E\left(x,2r,e_{n}\right)+r +Q^2D_E(x,2r)\Biggr),
\end{equation}
for every $\varphi\in C^{1}_{c}(\mathbf{D}_{r})$.
\end{proposition}

\begin{proof}
Upper and lower perimeter estimates established in \eqref{e:perimeterdens-1} and \eqref{e:perimeterdens} ensure that   in every cylinder \(C(x,4r)\) centered at \(x\in \partial E\), \(E\)    satisfies the assumption of Lemma \ref{lipappr} with a universal constant \(C=C(n,A)\) provided \(r\) is smaller than an universal radius \(\bar r\). This proves the first part of the proposition. The second part follows by  plugging in \eqref{e:precordes} \(F:=\psi_t(E)\), \(\psi_t(x)=x+t\varphi(\p x)e_n\), and by   performing the same computations  done in  \cite[Proof of Theorem 23.7]{Maggi12}. 
\end{proof}

\subsection{The Caccioppoli inequality}

By \eqref{quasih} one will deduce that under the assumption of Theorem \ref{excessimprove},  there exists an harmonic function \(h:D_r\to \R \) which is close to \(f\) in \(L^2\). This closeness, together with the regularity theory for harmonic function will allow to deduce the decay of an \(L^2\) type excess of \(f\) and thus for \(E\). In order to pass from the \(L^2\) excess to the classical one, one needs to esablish a Caccioppoli type inequality. To this end, given a set \(E\) and a vector \(\nu \in S^{n-1}\) we define the \emph{flatness} of $E$ at the point $x\in\R^{n}$, at the scale $r>0$ with respect to the direction $\nu\in\mathbb{S}^{n-1}$ as 
 \begin{equation}\label{flatness} 
\f_{E}(x,r,\nu):=\frac{1}{r^{n-1}}\inf_{h\in\R}\int_{\mathbf{C}
(x,r,\nu)\cap \partial^* E}\frac{|\nu\cdot(y-x)-h|^{2}}{r^{2}}\,d\mathcal{H}^{n-1}(y).
\end{equation}

\begin{proposition}[Caccioppoli inequality]\label{prop:cacc}
Let \(A>0\), and let  \(\beta, K, Q\) be  controlled by \(A\)  and \(R\ge 1\). Then there exists universal constants \(\eps_{\mathrm{cac}}\), and \(C_{\mathrm {cac}}\) such that if \(E\) is a minimizer of \eqref{e:problem}, \(x\in \partial E\),  and 
\[
r+\e_E(x,4r,e_n)\le \eps_{\mathrm{cac}},
\]
then
\begin{equation}\label{e:cacc}
\e_{E}(x,r,e_n)\le C_{\mathrm {cac}}\Bigl(\f_E(x,2r,e_n)+r+Q^2D_{E}(x,2r)\Bigr).
\end{equation}
\end{proposition}

\begin{proof}
The proof can be obtained by verbatim repeating the  arguments  of \cite[Chapter 24]{Maggi12} and using \eqref{e:precordes} instead of the perimeter minimality in the comparison estimate of \cite[Equation 24.48]{Maggi12}.

\end{proof}

\subsection{Dirichlet improvement}
We  now show that in the small excess regime there is  fixed scale decay of the Dirichlet energy.
\begin{proposition}[Decay of the Dirichlet energy]\label{p:dir}
Let \(A>0\), and let  \(\beta, K, Q\) be  controlled by \(A\)  and \(R\ge 1\). There exists a universal constant \(C_{\mathrm dir}>0\)   such that for all \(\lambda \in (0,1/2)\)  there exists  \(\eps_{\mathrm{dir}}=\eps_{\mathrm{dir}}(n,A, \lambda)\) satisfying the following:  if \(E\) is a minimizer of \eqref{e:problem}, \(x\in \partial E\) and 
\begin{equation}\label{e:epsdir}
r+\e_E(x,r,E_n)\le \eps_{\mathrm{dir}},
\end{equation}
then 
\[
D_E(x,\lambda r)\le C_{\mathrm{dir}}\lambda \Bigl(D_E(x, r)+r\Bigr).
\]
\end{proposition}

\begin{proof}
By \eqref{e:perimeterdens-1} and \eqref{e:perimeterdens} we have that if \(r\) is universally small we can apply Lemma \ref{fascia} to \(E\) in \(\C(x,r)\) to obtain a universal modulus of continuity \(\omega\) such that  for \(H=\{y: \q (y-x)\le 0\}\),
\[
\frac{|(E\Delta H)\cap B_{r/2}(x)|}{|B_{r/2}|}\le \omega\bigl(\eps_{\mathrm{dir}}\bigr).
\]
By Lemma \ref{1decay} (ii) (applied in \(B_{r/2}(x)\))  and the above inequality, for all \(\lambda \in (0,1/2)\) we can choose \(\eps_{\mathrm{dir}}=\eps_{\mathrm{dir}}(n, A, \lambda)\) sufficiently small  such that 
\begin{equation*}
\begin{split}
D_{E}(x,\lambda r)&\le C(n,A)\lambda \Biggl(D_{E}(x,\frac{r}{2})+\frac{r^3}{\lambda^{n}}\Biggr)
\\
&\le C(n,A)\lambda \Biggl(D_{E}(x,r)+\frac{\eps_{\mathrm{dir}}^2 r}{\lambda^{n}}\Biggr) \le C(n,A)\lambda \bigl(D_{E}(x,r)+ r\bigr),
\end{split}
\end{equation*}
where in the first inequality we have also exploited \eqref{e:rhobound} and  in the second  the obvious inequality \(D_E(x,r/2)\le 2^{n-1} D_E(x,r)\). This concludes the proof.
\end{proof}

\subsection{Excess improvement}

In this section we prove Theorem \ref{excessimprove}.

\begin{proof}[Proof of Theorem \ref{excessimprove}]  
We claim  that there exists a universal constant \(C_{\textrm{exc}}\) such that for all \(\lambda\in (0,1/8)\) there exists \( \eps_{\textrm{exc}}=\eps_{\textrm{exc}}(n, A,\lambda)\) satisfying the following:  for all minimizers of \eqref{e:problem} with \(\beta, K, Q\) controlled by \(A\) and \(R\ge 1\) if \(x\in \partial E\) the following holds
\[
\e_E(x,r)+D_{E}(x,r)+r\le  \eps_{\textrm{exc}}\quad\Longrightarrow\quad\e_E(x,\lambda r)\le C_{\mathrm{exc}} \lambda\Bigl(\e_E(x,r)+D_{E}(x,r)+r\Bigr).
\]
Note that  the above claim, combined with Proposition \eqref{p:dir} immediately implies the conclusion of the Theorem. Let us assume hence that there exists \(\lambda \in (0,1/8)\) a sequence of minimizers \(E_k\subset B_{R_k}\) with parameters \(\beta_k, K_k, Q_k\) controlled by \(A\),  radii \(r_k\)  and points \(x_k\in \partial E_k\) such that 
\[
\varepsilon_k=\e_{E_k}(x_k,r_k)+D_{E}(x_k,r_k)+r_k\to 0
\]
but 
\begin{equation}\label{ed0}
\e_{E_k}(x_k,\lambda r_k)\ge C_{\mathrm{exc}} \lambda \eps_k
\end{equation}
for a suitable universal constant \(C_{\mathrm{exc}}\). Note that up to translating and rotating we can assume that \(x_k=0\) and 
 that 
 \[
 \e_{E_k}(0,r_k)= \e_{E_k}(0,r_k,e_n).
 \]
We apply Proposition~\ref{lipappr2} to each $E_{k}$. Hence, there exists a sequence of $1$-Lipschitz functions $f_{k}:\R^{n-1}\to{}\R$ such that
\begin{subequations} 
\begin{gather}
\frac{\Hf^{n-1}\left(\C_{\frac{r_{k}}{2}}\cap\partial E_{k}\Delta \Gamma_{f_{k}}\right)}{r^{n-1}_{k}}\leq{}2^{n-1}C_{\mathrm{lip}}\varepsilon_{k}\,,
\label{e:lipmis}
\\
\frac{1}{r^{n-1}_{k}}\int_{\D_{r_{k}/2}}\left|\nabla{}f_{k}\right|^{2}\,dx\leq 2^{n-1}C_{\mathrm {lip}}\varepsilon_{k}\,,
\label{e:lipdir}
\\
 \|f_k\|_{\infty}\le \omega(\eps_k)r_k,
 \label{e:lipheig}
 \\
\Biggl|\fint_{\D_{\frac{r_{k}}{2}}} \nabla f_k\cdot\nabla \varphi\Biggr| \le 2^{n-1}C_{\mathrm {lip}} \|\nabla \varphi \|_{\infty} \varepsilon_k \qquad \text{for all \(\varphi\in C_c^1(\D_{\frac{r_k}{2}})\)}.
\label{e:lipharm}
\end{gather}
\end{subequations}
Let us set 
\begin{equation*}
g_{k}:=\frac{f_k^{r_k}-m_k}{\sqrt{\varepsilon_{k}}}\quad\text{where}\quad m_{k}:=\fint_{\D_{\frac{r_{k}}{2}}}\,f_{k}^{r_{k}}\,,\quad\text{and}\quad f_{k}^{r_{k}}(z):=\frac{f_{k}(r_{k}z)}{r_{k}}.
\end{equation*}
By the Poincar\'e -Wirtinger inequality and \eqref{e:lipdir}, 
\begin{equation}\label{e:univbound}
\sup_{k}\|g_{k}\|_{W^{1,2}(\D_{\frac{1}{2}})}\le 2^{n-1}C_{\mathrm{lip}}\,.
\end{equation}
Hence  there exists \(g\) in $W^{1,2}(\D_{\frac{1}{2}})$  such that  $g_{k}\rightharpoonup g$ weakly in $W^{1,2}(\D_{\frac{1}{2}})$ to some $g$ and  strongly in $L^{2}(\D_{\frac{1}{2}})$.
Moreover by \eqref{e:lipharm}, for all \(\varphi \in C_c^1(\D_{\frac{1}{2}})\)
\begin{equation} \label{ed2}
\begin{split}
\Bigl|\int_{\D_{\frac{1}{2}}} \nabla g \cdot \nabla \varphi\Bigr|&=\lim_{k\to+\infty}\frac{1}{\sqrt{\varepsilon_{k}}}\Bigr|\int_{\D_{\frac{1}{2}}}\nabla f_{k}^{r_{k}}\cdot\nabla \varphi\,dx\Bigr|
\\
&=\lim_{k\to+\infty}\frac{1}{\sqrt{\varepsilon_{k}}}\Bigr|\fint_{\D_{\frac{r_k}{2}}}\nabla f_{k}\cdot\nabla \varphi_{r_k}\,dx\Bigr|=0,
\end{split}
\end{equation}
where \(\varphi_{r_k}(z)=r_k\varphi(z/r_k)\in C_c^1(\D_{\frac{r_k}{2}})\) satisfies \(\|\nabla \varphi_{r_k}\|_{\infty}=\|\nabla \varphi\|_{\infty}\). 
Hence \(g\) is harmonic. By the mean value property and \eqref{e:univbound}
\[
\sup_{\D_{1/4}} |\nabla ^2 g|^2\le C(n)\int_{\D_{\frac{1}{2}}}|\nabla g|^2\le C(n,A).
\]
By Taylor expansion, 
\begin{equation}
|g(z)-g(0)-\nabla g(0)\cdot z|\le C(n,A)|z|^2\qquad\text{for all \( z\in{}\D_{\frac{1}{4}}\)}.
\end{equation}
If \(2\lambda \in (0,1/4)\) we can integrate the above inequality to get 
\begin{equation*}
\fint_{\D_{2\lambda}}\left|g(z)-g(0)-\nabla g(0)\cdot z\right|^{2}\,dz\leq{} C(n,A)\,\lambda^{4}. 
\end{equation*}
Recall that, by the mean value property of harmonic functions, for every $r\leq{}\frac{1}{2}$ we have
\begin{equation*}
(g)_{r}:=\fint_{\D_{r}}g\,dx=g(0)\;\;\text{ and }\;\;\left(\nabla g\right)_{r}=\nabla g(0).
\end{equation*}
Hence,
\begin{equation*} 
\begin{split}
\lim_{k\to+\infty}\fint_{\D_{2\lambda}}|g_{k}(z)-(g_{k})_{2\lambda}-\left(\nabla g_{k}\right)_{2\lambda}\cdot z|^{2}\,dz &=\fint_{\D_{2\lambda}}|g(z)-(g)_{2\lambda}-\left(\nabla g\right)_{2\lambda}\cdot z|^{2}dx\\
&=\fint_{\D_{2\lambda}}\left|g(z)-g(0)-\nabla g(0)\cdot z\right|^{2}dx\\
&\leq C(n,A)\lambda^{4}.
\end{split}
\end{equation*}
which, by the definition of \(g_k\) and changing variables implies
\begin{equation} \label{ed6}
\lim_{k\to+\infty}\frac{1}{\varepsilon_{k}(\lambda r_{k})^{n+1}}\int_{\D_{2\lambda r_{k}}}\left|f_{k}(z)-(f_{k})_{2\lambda r_{k}}-\left(\nabla f_{k}\right)_{2\lambda r_{k}}\cdot z\right|^{2}\,dx\leq C(n,A)\,\lambda^{2}.
\end{equation}
Let us define
\begin{equation*}
\nu_{k}:=\frac{\left(-\left(\nabla f_{k}\right)_{2\lambda r_{k}},1\right)}{\sqrt{1+\left|\left(\nabla f_{k}\right)_{2\lambda r_{k}}\right|^{2}}}\qquad h_{k}:=\frac{\left(f_{k}\right)_{2\lambda r_{k}}}{\sqrt{1+\left|\left(\nabla f_{k}\right)_{2\lambda r_{k}}\right|^{2}}}\,,
\end{equation*}
and note that, by \eqref{e:lipdir}, Jensen inequality and \eqref{e:lipheig}
\begin{equation}\label{e:est}
|\nu_k-e_n|^2\le C\Bigl(\fint_{\D_{\lambda r_k}}|\nabla f_k|\Bigr)^2\le C(n, A,\lambda)\eps_k\qquad\text{and}\qquad |h_k|\le C \omega(\eps_k)r_k.
\end{equation}
Since the $f_{k}$'s are $1$-Lipschitz, \eqref{ed6} implies
\begin{equation*} \label{ed7}
\begin{split}
\limsup_{k\to +\infty} &\frac{1}{\varepsilon_{k}(\lambda r_{k})^{n+1}}\int_{\Gamma_{f_{k}} \cap \C_{2\lambda r_{k}}}\,\left|\nu_{k}\cdot x-h_{k}\right|^{2}\,d\Hf^{n-1}(x)
\\
&\leq \lim_{k\to +\infty}\frac{\sqrt{2}}{\varepsilon_{k}(\lambda r_{k})^{n+1}}\int_{\D_{2\lambda r_{k}}}\left|f_{k}(z)-(f_{k})_{2\lambda r_{k}}-\left(\nabla f_{k}\right)_{ 2\lambda r_{k}}\cdot z\right|^{2}\,dz\leq C(n,A)\,\lambda^{2}.
\end{split}
\end{equation*}
and thus
\begin{equation}\label{e:flat1}
\limsup_{k\to +\infty} \frac{1}{\varepsilon_{k}(\lambda r_{k})^{n+1}}\int_{\Gamma_{f_{k}}\cap \partial E_k \cap \C_{2\lambda r_{k}}}\,\left|\nu_{k}\cdot x-h_{k}\right|^{2}\,d\Hf^{n-1}(x)\le C(n,A)\,\lambda^{2}.
\end{equation}
On the other hand, \eqref{e:lipmis}, Lemma \eqref{fascia} and \eqref{e:est} imply  
\begin{equation}\label{e:flat2}
\begin{split}
\frac{1}{\varepsilon_{k}(\lambda r_{k})^{n+1}}&\int_{( \partial E_k \setminus \Gamma_{f_{k}})\cap \C_{2\lambda r_{k}}}\left|\nu_{k}\cdot x-h_{k}\right|^{2}\,d\Hf^{n-1}(x)
\\
&\le C(n, A, \lambda)\frac{\Hf^{n-1}\bigr((\partial E_{k}\Delta \Gamma_{f_k})\cap \C_{r_{k}}\bigr)}{\eps_k r_{k}^{n-1}}\Biggl(|\nu_k-e_n|^2+\sup_{x \in \partial E_k\cap \C_{r_k}} \frac{|\q x|}{r^2_k}+\frac{|h_k|^2}{r^2_k}\Biggr)
\\
&\le C(n, A, \lambda)\frac{\Hf^{n-1}\bigr((\partial E_{k}\Delta \Gamma_{f_k})\cap \C_{r_{k}}\bigr)}{\eps_k r_{k}^{n-1}}\bigl(\eps_k+\omega (\eps_k)\bigr)=o(1).
\end{split}
\end{equation}
Combining \eqref{e:flat1} and \eqref{e:flat2} we deduce that
\begin{equation}\label{e:flat3}
\begin{split}
\limsup_{k\to \infty}&\frac{\f_{E_k}(0,2\lambda r_k, \nu_k)}{\eps_k} 
\\
&\le \limsup_{k\to \infty}\frac{1}{\eps_k(\lambda r_k)^{n+1}}\int_{\partial E_k \cap \C_{2\lambda r_{k}}}\,\left|\nu_{k}\cdot x-h_{k}\right|^{2}\,d\Hf^{n-1}(x)\le C(n,A)\,\lambda^{2}.
\end{split}
\end{equation}
On the other hand, by the perimeter density estimates \eqref{e:perimeterdens-1} and \eqref{e:est}
\begin{equation*}
\begin{split}
\e_{E_k}&(0,4\lambda r_k, \nu_k)\le \frac{1}{(4\lambda r_{k})^{n-1}}\int_{\partial E_{k}\cap \C_{4\lambda r_{k}}}\,\frac{\left|\nu_{E_{k}}-\nu_{k}\right|^{2}}{2}\,d\Hf^{n-1}
 \\
&\leq C(n, \lambda) \Biggl(\e_{E_k}(0, r_k, e_n)+
+|e_{n}-\nu_{k}|^{2}\frac{P(E, B_{ r_k})}{ r_{k}^{n-1}}\Biggr)=o(1).
\end{split}
\end{equation*}
Hence we can  apply Proposition \ref{prop:cacc} in \(B_{4\lambda r_k}\) to get that
\begin{equation}\label{e:flatness}
\begin{split}
\e_{E_k}(0,\lambda r_k)&\le \e_{E_k}(0,\lambda r_k, \nu_k)
\\
&\le C_{\mathrm {cac}} \Bigl(\f_{E_k}(0,2\lambda r_k, \nu_k)+Q^2D_{E_k}(0,2\lambda r_k)+\lambda r_k\Bigr),
\end{split}
\end{equation}
where in the first inequality we have used \eqref{e:exsexc}. Furthermore, by Proposition \eqref{p:dir} applied in \(B_{r_k}\) we have
\begin{equation}\label{e:dird}
Q^2D_{E_k}(0, 2\lambda r_k)\le C_{\mathrm{dir}}\lambda (Q^2D_{E_k}(0, r_k)+Q^2r_k)\le C(n,A)\lambda \eps_k. 
\end{equation}
Combining \eqref{e:flat3}, \eqref{e:flatness} and \eqref{e:dird} we thus infer that 
\[
\limsup_{k\to \infty} \frac{\e(0,\lambda r_k)}{\eps_k}\le  C(n, A) \lambda\,,
\]
in contradiction with \eqref{ed0} if \(C_{\mathrm{exc}}\) is chosen big enough depending only on  \(n\) and \(A\).
\end{proof}
\section{Proof of Theorems~\ref{thm:main} and~\ref{thm:maineps}}\label{s:thmmain}

In this section we prove our main theorems, Theorem \ref{thm:maineps} is an immediate consequcne of the following slightly more general theorem.

\begin{theorem}\label{thm:mainmeglio}
Let \(A>0\) \(\vartheta \in (0,1)\), and let  \(\beta, K, Q\) be  controlled by \(A\)  and \(R\ge 1\). There exist  constants \(C_{\mathrm {reg}}(n, A, \theta)>0$  and \(\varepsilon_{\mathrm {reg}}=\varepsilon_{\mathrm {reg}}(n, A, \theta)>0\)     if \(E\) is a minimizer of \eqref{e:problem}, \(x\in \partial E\) \(r>0\) and \(\nu\in S^{n-1}\) are such that 
\begin{equation*}
 r+Q^2D_{E}(x,2r)+\e_E(x,2r,\nu)\leq{}\varepsilon_{\mathrm{reg}},
\end{equation*}
then there exists a \(C^{1,\vartheta} \) function  \(f:\R^{n-1}\to \R\) with
\footnote{ Here
\[
[\nabla f]_{\vartheta/2}:=\sup_{x\ne y } \frac{|\nabla f(x)-\nabla f(y)|}{|x-y|^\frac{\vartheta}{2}}.
\]
}
\[
f(0)=0\,, \quad |\nabla f(0)-\nu|^2+ {r^\vartheta} [\nabla f]^2_{\vartheta/2} \le C_{\mathrm {reg}}\bigl( r+Q^2D_{E}(x,2r)+\e_E(x,2r,\nu)\bigr)\,,
\]
such that
\[
E\cap B_{r}(x)=\Bigl\{y\in B_{r}(x): \nu\cdot (y-x)\le  f(\p^\nu (y-x))\Bigr\}.
\]
\end{theorem}

\begin{proof}
Given \(\vartheta \in (0,1)\) we fix  \(\bar \lambda \in (0,1/8)\) be such that 
\begin{equation}\label{e:barlambda}
 C_{\mathrm {dec}} \bar \lambda+\bar \lambda \le {\bar\lambda}^{\vartheta},
 \end{equation}
 and we let \(\bar \eps\) be the corresponding \(\eps_{\mathrm {dec}}\) in Theorem \ref{excessimprove}. Note that \(\bar \eps\) depends only on \(n\), \(A\) and \(\vartheta\). We now choose \(\varepsilon_{\mathrm {reg}}\) so that for all \(y\in \partial E\cap B_{r}(x)\)
\[
\begin{split}
 r+Q^2D_{E}(y,r)&+\e_E(y,r)\le  r+Q^2D_{E}(y,r, \nu)+\e_E(y,r,\nu)
 \\
 &\le 2^{n-1}\bigl( r+Q^2D_{E}(x,2r,\nu)+\e_E(x,2r,\nu)\bigr)\le 2^{n-1} \varepsilon_{\mathrm {reg}}\le \bar \eps.
\end{split}
\]
Hence we can apply Theorem  \ref{excessimprove} and \eqref{e:barlambda} to deduce that for all \( y\in \partial E\cap B{r/2}(x)\),
\[
\bar  \lambda r+Q^2D_{E}(y,\bar\lambda r)+\e_E(y,\bar\lambda r)\le \bar \lambda^{\vartheta} \bigl(r+Q^2D_{E}(y,r)+\e_E(y,r)\bigr).
\]
Iterating we get
\[
\e_E(y,\bar\lambda^{k} r)\le {\bar \lambda}^{k\vartheta}\bigl(r+Q^2D_{E}(y,r)+\e_E(y,r)\bigr),
\]
which implies
\[
\e_E(y, s)\le C(\vartheta)\Bigl( \frac{s}{r}\Bigr)^{\vartheta}\bigl(r+Q^2D_{E}(y,r)+\e_E(y,r)\bigr) \qquad\text{for all \(s\le r\).}
\]
By classical arguments this together with the  density  estimates \eqref{e:perimeterdens-1} and  \eqref{e:perimeterdens}, implies that for all \(y\in B_{r}(x)\cap \partial E\) there exists \(\nu_{y}\) such that 
\[
\e_{E}(y,s/2,\nu_{y})=C(n,\vartheta,A) \Bigl( \frac{s}{r}\Bigr)^{\vartheta}\bigl(r+Q^2D_{E}(y,r)+\e_E(y,r)\bigr) \qquad\text{for all \(s\le r\).}
\]
and 
\[
|\nu_{y}-\nu|^2\le C(n,A)\bigl(r+Q^2D_{E}(y,2r, \nu)+\e_E(y,r)\bigr),
\]
The last two display   yield the  desired conclusion, see for instance \cite[Theorem 26.3]{Maggi12} or \cite[Theorem 4.8]{Giusti83}.
\end{proof}

We can now prove Theorem \ref{thm:main} by following the arguments \cite{FuscoJulin15} 
\begin{proof}[Proof of Theorem \ref{thm:main}]
By Theorem \ref{thm:maineps}, if we set 
\[
\Sigma_{E}=\bigl\{x\in \partial E: \limsup_{r\to 0} \e_{E}(x,r)+D_{E}(x,r)>0\big\}\,,
\]
then \(\partial E\setminus \Sigma_E\) is a \(C^{1,\vartheta}\) manifold for all \(\vartheta\in (0,1/2)\). Hence we will conclude the proof if we show that 
\begin{equation*}\label{e:fine}
\Hf^{n-1-\eta}(\Sigma_{E})=0,
\end{equation*}
for some \(\eta=\eta(n,B)>0\). Recall that by Lemma \ref{l:higherintegrability}, \(|\nabla u_E|^{2p}\in L_{\mathrm{loc}}^1\) for some \(p=p(n,B)>1\), hence, by H\"older inequality
\[
\Sigma^1_{E}=\bigl\{x: \limsup_{r\to 0} D_{E}(x,r)>0\bigr\}\subset\Biggl\{ x: \limsup_{r\to 0} \frac{1}{r^{n-p}}\int_{B(x,r)}|\nabla u_E|^{2p}>0\Biggr\}.
\]
Hence, by~\cite[Theorem 2.10]{EvansGariepy15}, \(\Hf^{n-p}(\Sigma^1_{E})=0\). We  now  show that  
\[
\Hf^{\alpha}(\Sigma_E\setminus \Sigma_E^1)=0 
\]
for all \(\alpha>n-8\) which clearly concludes the proof. Let us fix \(\alpha>n-8\) and assume the contrary. By~\cite[Proposition 11.3]{Giusti83}, there will be a point 
\[
x\in \Sigma^2_E:=\Big\{x\in \partial E: \limsup_{r\to 0} \e_E(x,r)>0\,, \lim_{r\to 0} D_E(x,r)=0\Bigr\},
\]
and a sequence \(r_k\to 0\) such that 
\[
\limsup_{k\to \infty} \frac{\Hf_{\infty}^\alpha(\Sigma^2_E\cap B(x,r_k))}{r_k^\alpha}\ge c(\alpha)>0.
\]
where \(\Hf_{\infty}^s\) is the infinity Hausdorff pre-measure.
Let us set \(E_{k}=(E-x)/r_k\) and note that  by \eqref{e:perimeterdens-1} and the above equation
\[
P(E_k,B_s)\lesssim s^{n-1}\qquad,
\]
for all \(s>0\) and
\begin{equation}\label{e:haus}
\limsup_{k\to \infty} \Hf_{\infty}^{\alpha}(\Sigma^2_{E_k}\cap B_1)\ge c(\alpha),
\end{equation}
where \(\Sigma^2_{E_k}=(\Sigma_{E}^2-x)/r_k\). Up to subsequences, \(E_k\to F\). We claim that  \(F\) is a local minimizer of the perimeter. Indeed if \(G\Delta F\Subset B_{s}\), by averaging we  choose \(t\in (s,2s)\) such that 
\[
\Hf^{n-1} ((E_k\Delta G)\cap \partial B_{t})= \Hf^{n-1} ((E_k\Delta F)\cap \partial B_{t})\le \frac{|(E_k\Delta F)\cap B_{2s}|}{s}=\sigma_k\to 0.
\]
With this choice, defining \(G_k=(x+r_kG)\cap B_{r_kt}(x)\cup (E\setminus B_{tr_k}(x))\)  and note that \(E\Delta G_k\Subset B_{2s r_k}(x)\). Hence  by  \eqref{e:precordes} and classical computations
\[
\begin{split}
P(F,B_t)-P(G,B_t)&\le\limsup_{k\to \infty} \frac{P(E_k, B_{tr_k}(x))-P(G_k, B_{tr_k}(x))}{r_k^{n-1}}
\\
&\lesssim \limsup_{k\to \infty}  \sigma_{k}+s^nr_k+s^{n-1} D_{E}(x,sr_k)=0,
\end{split}
\]
which implies the desired minimality property. Moreover, by using \(G=F\) we also deduce that \(P(E_k,B_s)\to P(F, B_s)\) for almost all  \(s>0\). 

Let now \(\Sigma_F\) be the singular set of \(F\), and recall that, by the  regularity theory for set of minimal perimeter \cite[Part III]{Maggi12}, \(\Hf^{\alpha}(\Sigma_{F})=\Hf_\infty^{\alpha}(\Sigma_{F})=0\). Hence by the  definition of Hausdorff measure,  for all \(\delta>0\) there exists an open set \(U_\delta\) such that
\[
\Sigma_F\cap B_2\subset U_{\delta}\qquad\text{and}\qquad \Hf_\infty^\alpha (U_\delta)\le \delta.
\]
We claim that there exists \(k=k_\delta>0\) such that \(\Sigma^2_{E_k}\cap B_1\subset U_{\delta}\) which will be in contradiction with \eqref{e:haus} if \(\delta\) is chosen small enough. Assume the claim is false, hence there is  a sequence of points \(\Sigma^2_{E_k}\cap B_1 \ni y_k\to \bar y\in \overline{B_1}\) with  \(\dist (\bar y, \Sigma_F)>0\). It is easy to see that, by the lower perimeter estimates \eqref{e:perimeterdens}, \(\bar y\in \partial F\). Hence by regularity, for all \(\eps>0\) there exists \(r>0\) such that
\[
\e_{F}(\bar y,r)\le \eps.
\]
By perimeter convergence, this implies that, for \(k\) large
\[
\e_{E}(x+r_ky_k,rr_k)=\e_{E_k}(y_k, r) \le \e_{F}(\bar y,r)+\eps\le 2\eps.
\]
Choosing \(\eps\ll1\) we can apply Theorem \ref{thm:maineps} to deduce that \(x+r_ky_k\notin \Sigma_{E}^2\), i.e. \(y_k\notin \Sigma^2_{E_k}\). This final contradiction concludes the proof.

\end{proof}


\end{document}